\documentclass[10pt, a4paper, reqno]{amsart}

\usepackage{tikz-cd}
\usepackage{array}

\usetikzlibrary{snakes}

\usepackage{amsmath, amssymb, amsfonts, amstext, amsthm, amscd}
\usepackage[mathscr]{euscript}
\usepackage{enumerate}

\usepackage{float}
\usepackage{graphicx}

\usepackage[raiselinks=false,colorlinks=true,citecolor=blue,urlcolor=blue,linkcolor=blue,bookmarksopen=true,pdftex, backref=page]{hyperref}

\usepackage[paperheight=11.4in,paperwidth=8.1in,bindingoffset=0in,left=0.5in,right=0.5in,top=0.5in,bottom=0.5in,headsep=.5\baselineskip]{geometry}
\usepackage{tikz}
\usepackage{tikz-cd}
\usetikzlibrary{snakes}
\usetikzlibrary{intersections, calc}
\usetikzlibrary{decorations.markings}%for middle arrow
\usetikzlibrary{arrows.meta}% arrow tip library
\usetikzlibrary{bending}% better arrow head for bended lines

\usepackage[english]{babel}
\usepackage{chngcntr}

\newcommand\Sym{\operatorname{Sym}}

\newcommand\tensor{\otimes}

\newcommand\tesnor{\otimes}

\newcommand\tr{\operatorname{tr}}

\newcommand\Spec{\operatorname{Spec}}

\newcommand\End{\operatorname{End}}

\usepackage{amsmath, amssymb, amsfonts, amstext, amsthm, amscd}

\usepackage{mathrsfs}

\newcommand{\OU}{{\mathscr{O}_U}}
\newcommand{\OS}{{\mathscr{O}_S}}

\newcommand\bq{\begin{equation}}
\newcommand\eq{\end{equation}}

\newtheorem{proposition}{Proposition}[section]
\newtheorem{theorem}[proposition]{Theorem}
\newtheorem{corollary}[proposition]{Corollary}
\newtheorem{example}[proposition]{Example}

\newtheorem{lemma}[proposition]{Lemma}
\newtheorem{definition}[proposition]{Definition}
\newtheorem{rmk}[proposition]{\textbf{Remark}}

\theoremstyle{remark}

\usepackage{fullpage}
\begin{document}

\title[Classifying Binary Quadratic Forms using Clifford Invariants]{Classifying Binary Quadratic Forms using Clifford Invariants}

\author[S. Mondal]{Soham Mondal}
\address{Department of Mathematics, Indian Institute of Technology Madras, Chennai 600036, India}
\email{getsoham1@gmail.com}
    
\author[T. E. V. Baalaji]{T. E. Venkata Balaji}
\address{Department of Mathematics, Indian Institute of Technology Madras, Chennai 600036, India}
\email{tevbal@iitm.ac.in}

	
    
\begin{abstract}
We functorially identify similarity classes of line-bundle-valued quadratic forms on rank two vector bundles with isomorphism classes of pairs consisting of the degree zero and the degree one parts of the associated generalized Clifford algebras. As applications, we generalize the Gauss Composition and explore connections with Picard groups of quadratic algebras.
\end{abstract}

\maketitle

\section{Introduction}
In the book \cite[Chapter V, Proposition 2.4.1, \S2]{knus}, M. A. Knus proved that if $(E, q)$ and $(E', q')$ are two $R$-valued quadratic modules of rank 2 over a ring $R$, then $(E, q)$  and $(E', q')$ are isometric if and only if their full Clifford algebras $C(E, q)$ and $C(E', q')$ are isomorphic as graded algebras. Further, he proved that for the rings $R$ where Witt cancellation holds and when the two quadratic modules $(E, q)$ and $(E', q')$ are nondegenerate, $(E, q)$  and $(E', q')$ are isometric if and only if their full Clifford algebras $C(E, q)$ and $C(E', q')$ are isomorphic as $R$-algebras (though they may not be isomorphic as graded algebras) and their even Clifford algebras $C_0(E, q)$ and $C_0(E', q')$ are isomorphic as $R$-algebras. However, in the context of non-trivial line bundle-valued quadratic forms, the full generalized Clifford algebra does not, in general, possess a natural algebra structure; see \cite{bichselknus}  for a comprehensive discussion of this phenomenon. As a result, one cannot anticipate a classification up to isometries analogous to that available in the trivial line bundle-valued case. Nevertheless, the even-degree component of the full generalized Clifford algebra associated to a line bundle-valued quadratic form $(E, q, L)$, denoted $C_0(E, q, L)$, does admit a natural algebra structure, while the odd-degree component, denoted $C_1(E, q, L)$, acquires a canonical structure as a bimodule over $C_0(E, q, L)$. For a binary quadratic form $(E, q, L)$ over an arbitrary base scheme $S$, the generalized even Clifford algebra, $C_0(E, q, L)$, is a quadratic $\OS$-algebra. Our main result, which uses definitions explained later, is Theorem \ref{MainTheorem}. It can be stated as follows:

\textbf{Theorem 3.1.}   For any scheme $S$, the natural map
    \[(E, q, L)\mapsto (C_0(E, q, L), C_1(E, q, L))\]
    induces a bijective correspondence
\[
\left\{
\parbox{2.5in}{\centering Similarity classes of binary quadratic forms (E, q, L) over \( S \)}
\right\}
\longleftrightarrow
\left\{
\parbox{2.5in}{\centering Isomorphism classes of pairs \( (C, E) \), \\ with \( C \) a quadratic algebra over \( S \), \\ and \( E \) a \textit{traceable} \( C \)-module}
\right\},
\]
which is functorial in $S$, and which preserves discriminants up to sign, namely satisfies  $$
\Delta(q) = -\Delta(C_0(E, q, L)).
$$
In the above, $E$ is a traceable $C$-module means that Zariski-locally on $S$, the trace map on $C$ equals the trace map on $E$ as a left $C$-module, a notion defined and employed by M. M. Wood in \cite{Wood}.
An isomorphism of pairs $(C , E)$ and $(C', E')$ is given by an isomorphism $C \cong C'$ of $\OS$-algebras, and an isomorphism $E \cong E'$ of $\OS$-modules that respects the $C$ and $C'$ module structures. 

As an immediate consequence of Theorem \ref{MainTheorem}, we obtain the following result:

\textbf{Corollary 3.4.}
Let \( q_1: E_1 \to L_1 \) and \( q_2: E_2 \to L_2 \) be two binary quadratic forms over a scheme \( S \). If there exist:

\begin{enumerate}
    \item An isomorphism between \( C_0(E_1, q_1, L_1) \) and \( C_0(E_2, q_2, L_2) \) as \( \OS \)-algebras, and
    \item An isomorphism between \( E_1 \) and \( E_2 \) as \( \OS \)-modules,  with the \( C_0(E_1, q_1, L_1) \)-module (and \( C_0(E_2, q_2, L_2) \)-module) structures compatible with the isomorphism in (1) above,
\end{enumerate}
then \( q_1 \) and \( q_2 \) are similar. Conversely, the similarity of \( q_1 \) and \( q_2 \) implies the existence of such isomorphisms.

So in this work, we significantly extend the scope of Knus’s results by providing, in Theorem~\ref{MainTheorem}, a classification up to similarity of binary quadratic forms in the following broader contexts:

\begin{itemize}
    \item \textit{Line bundle-valued binary forms}: We extend the theory to encompass forms whose values lie in invertible sheaves (line bundles) rather than merely in the structure sheaf $ \OS $.
    
\item  \textit{Degenerate forms}: Unlike previous work that primarily focused on nondegenerate forms, we include degenerate forms in our classification, thereby addressing a broader and more general class of quadratic forms.
    
    \item  \textit{General base schemes $ S $:} Our results apply to arbitrary base schemes $ S $. 
\end{itemize}

This broader framework allows for a more flexible and geometrically meaningful treatment of binary quadratic forms, enhancing their applicability in both algebraic and geometric contexts.

After restricting the bijection in Theorem \ref{MainTheorem} to the primitive binary quadratic forms, we have the following result:\\
\textbf{Theorem 3.12.}  For any scheme $S$, the natural map
    \[(E, q, L)\mapsto (C_0(E, q, L), C_1(E, q, L))\]
    induces a bijective correspondence
\[
\left\{
\parbox{2.5in}{\centering Similarity classes of primitive binary quadratic forms (E, q, L) over \( S \)}
\right\}
\longleftrightarrow
\left\{
\parbox{2.5in}{\centering Isomorphism classes of pairs \( (C, E) \), \\ with \( C \) a quadratic algebra over \( S \), \\ and \( E \) a locally free rank 1 \( C \)-module}
\right\},
\]
which is functorial in $S$, and which preserves discriminants up to sign, namely $$
\Delta(q) = -\Delta(C_0(E, q, L)).
$$ 

While Theorem \ref{classifi1} furnishes a classification of primitive binary quadratic forms over an arbitrary base scheme $S$, building upon the foundational results of Theorem \ref{MainTheorem}, we offer here an alternative perspective that extends the theory of composition and structural analysis originally developed by Kneser \cite{KNESER} and subsequently refined by Bichsel and Knus \cite{bichselknus}. In their work, Bichsel and Knus established a classification of non-degenerate, primitive binary quadratic forms over an arbitrary ring $R$, valued in an invertible $R$-module, via the construction of the universal norm form (see \cite[Example 4.2]{bichselknus}). In this paper, we broaden and generalize their framework in several significant respects. In particular, we consider: binary quadratic forms taking values in arbitrary line bundles over a general base scheme $S$, which may not necessarily be non-degenerate.
% \begin{enumerate}
%     \item A general base scheme $S$, rather than an affine base;
%     \item Binary quadratic forms taking values in arbitrary line bundles over $S$; and
%     \item Quadratic forms that are not necessarily non-degenerate—that is, we allow for possibly degenerate forms.
% \end{enumerate}
Thus, we generalize the results of Knus \cite{knus} and Bichsel-Knus \cite{bichselknus}.

As an application of Theorem \ref{MainTheorem}, we obtain the following structural result, which gives a deep connection between quadratic algebras and the theory of Clifford algebras associated with binary forms:\\
\textbf{Theorem 5.10.}  Given a quadratic $\OS$-algebra $C$ over an arbitrary base scheme $S$, there exists a binary quadratic form with values in a suitable line bundle, whose generalized even Clifford algebra is isomorphic to $C$.\\
In light of the proof of Theorem \ref{existence of form}, a natural question emerges: Given a quadratic algebra $ C $ over the base scheme $ S $, how can one systematically parameterize its Picard group in terms of quadratic forms? Building upon our principal theorem, we were able to demonstrate the following:

\textbf{Theorem 5.11.} Let $ S $ be a general scheme and  $ C $ be a quadratic algebra over $ S $. The natural map 
\[(E, q, L)\rightarrow C_1(E, q, L)\]
induces
\[
\left\{
\begin{array}{c}
\text{Similarity classes of primitive binary quadratic forms (E, q, L) over } S \\
\text{having even Clifford algebra $C_0(E, q, L)$ isomorphic to } C
\end{array}
\right\}
\longleftrightarrow \mathrm{Pic}(C) / \sim,
\]
where for $ E, E' \in \mathrm{Pic}(C) $,  $ E \sim E' $ if and only if there exists an automorphism $ \varphi $ of $ C $ over $ S $ such that $ E' \cong \varphi^*E $ as $ C $-modules. Here, $ \varphi^*E $ denotes the $ C $-module $ E $ whose structure is locally given by $ \alpha \cdot e = \varphi(\alpha)e $, and  $\mathrm{Pic}(C) / \sim$ denotes the set of $\sim$-equivalence classes in $\mathrm{Pic}(C)$.

Theorem \ref{Quotient Picard}, provides an approximation to the parametrization of the Picard group of $C$, with the primary obstruction arising from the existence of nontrivial automorphisms of $C$. To resolve this obstruction, we introduce a rigidification of quadratic algebras by eliminating all nontrivial automorphisms via the notion of an \emph{orientation}. By an \emph{$ N $-twisted binary form}, we mean a binary form taking values in a specified line bundle associated with a fixed invertible sheaf determined by $ N $. This refinement leads to the following result:

\textbf{Theorem 5.17.} Let $S$ be a general scheme such that $2$ is not a zero divisor in $\Gamma(S, \OS)$. Let $C$ be a quadratic algebra over $S$. There exists a set-theoretic bijection between:
\[
\left\{
\begin{array}{c}
\text{Similarity classes of primitive } N\text{-twisted binary quadratic forms over } S, \\
\text{whose even Clifford algebra is isomorphic to } C, \\
\text{and which preserve a fixed } N\text{-orientation}
\end{array}
\right\}
\longleftrightarrow{ \operatorname{Pic}(C)}.
\]

In 2011, Wood, in her paper \cite[Theorem 1.4]{Wood}, established a functorial bijection between the similarity classes of \textit{linear} binary quadratic forms over $ S $ and the isomorphism classes of pairs $ (C, E) $, where $ C $ is a quadratic algebra and $ E $ is a \textit{traceable} $ C $-module. However, it is important to draw a crucial distinction between Wood's definition of \textit{linear} binary quadratic forms and the classical notion of binary quadratic forms used above and as formulated for e.g., by Kneser \cite{KNESER}. Specifically, Wood's \textit{linear} binary quadratic forms on rank 2 vector bundles correspond precisely to the classical binary quadratic forms defined on the \textit{dual} vector bundles. For further details, see \cite[Proposition 6.1]{Wood}. This correspondence highlights a subtle but essential connection between these two perspectives.

The results of derived categories are employed in the proof of Theorem 1.4 in \cite{Wood}, where they play a crucial role in the construction of the associated algebra and module structures via hypercohomology. Specifically, the complexes arising from the \textit{linear} binary quadratic form are treated within the derived category of sheaves on the base scheme $S$, allowing for a coherent and base-change-resilient definition of the associated quadratic algebra and traceable module. A detailed exposition of this approach, including the necessary background on derived categories and their application to the moduli problems at hand, can be found in \cite[\S3 and Appendix A]{Wood}. Our focus diverges not only in the fact that we replace Wood's \textit{linear} binary quadratic forms with classical binary quadratic forms, but we also in our aim to address the classification problem using the lens of Clifford algebras, which provide a rich algebraic structure encoding the intrinsic properties of classical quadratic forms.

Given a binary quadratic form $ q \colon E \to L $ in the classical sense, there is \emph{a priori}, no canonical method to construct a binary quadratic form on the dual module $ E^\vee $ within the same framework. However, by employing Wood's construction together with the main result of this paper, Theorem~\ref{MainTheorem}, we are able to establish such a construction in a natural and well-defined manner.

Using the  bijection provided in \cite[Proposition 6.1]{Wood}  and the correspondence of Theorem \ref{MainTheorem} together with Wood's correspondence \cite{Wood}, we establish the following:

\textbf{Theorem 5.4.} Let  $E$ be locally free rank 2, $L$ be locally free rank 1 over a general scheme $S$. Then we have
    \begin{itemize}
        \item  A duality between classical binary quadratic forms on $E$ and $E^\vee$, with values in $L$;
        \item A duality between Wood's \textit{linear} binary quadratic forms on $E$ and $E^\vee$, with values in $L$;
        \item Involutions on the space of Clifford pairs and on the space of Wood's pairs.
    \end{itemize}

Projective duality is a fundamental concept in projective geometry that establishes a symmetry between points and lines (or more generally, geometric objects of complementary dimensions). Now, over an algebraically closed field $ k $ with $\mathrm{char}(k) \neq 2$, the algebraic duality between classical and Wood forms corresponds to this geometric duality:

 \begin{enumerate}
     \item Classical binary quadratic forms define conics.
    \item Wood's \textit{linear} binary quadratic forms define their dual conics.

 \end{enumerate}
     
Thus:
\begin{quote}
    ``Over algebraically closed fields of characteristic not 2, algebraic duality implies geometric duality, and vice versa.''
\end{quote}

In his paper \cite[Theorem 3.24]{dallaporta}, William Dallaporta provided a parametrization of the Picard group associated with a given quadratic algebra $ C $. However, his approach relies on constructions developed in Wood's work \cite{Wood}, which uses \textit{linear} binary quadratic forms and moreover does not incorporate the Clifford-theoretic framework.  Consequently, Dallaporta's parametrization is rooted in \textit{linear} binary quadratic forms (see  \cite[Definition 3.12]{dallaporta}), which differ substantially from our definition of classical binary quadratic forms. Furthermore, his methodology lacks the  Clifford-theoretic perspective, which is central to our analytical framework. This absence highlights a profound divergence in both conceptual underpinnings and technical methodology.

The composition of binary quadratic forms is a long-established topic in number theory. Since Gauss’s seminal work in his \emph{Disquisitiones Arithmeticae}~\cite{Gauss}, there have been numerous efforts to simplify and generalize this concept. Martin Kneser, in his influential work~\cite{KNESER}, introduced the notion of quadratic modules \(E\) as modules over their even Clifford algebra \(C_0(E)\).  Kneser's key insight is there to interpret composition as tensor product over $C_0(E)$. This provides a more elegant and general framework than previous approaches. Notably, Kneser succeeded in establishing this composition law for affine schemes without imposing any conditions on the characteristic of the base ring. Kneser works with the case where the even Clifford algebras are identical. Specifically, in Section 6 of his article, he examines the case where $C_1 = C_2 = C$ and the isomorphisms $\phi_1$, $\phi_2$ are the identity map. In this case, the composition map is called ``of type $C$". In \cite[Theorem 3]{KNESER}, Kneser shows that the isomorphism classes of primitive binary quadratic forms of type $C$ with composition of type $C$ form an abelian group $G(C)$. Furthermore, Kneser establishes an isomorphism between the groups $H(C)$ and $\mathrm{Pic}(C)$, where the group $H(C)$ is defined using composition of type $C$. For a detailed account of this construction, see \cite[Proposition 2, \S6]{KNESER}. In \cite{KNESER}, Kneser considered the isomorphism classes of quadratic modules of type $C$, whereas we used similarity classes of binary quadratic forms. In this article, we consider a more general composition law. We consider the composition of two primitive binary quadratic forms having only isomorphic even Clifford algebras; in particular, the isomorphisms may not be equal to the identity map. Our approach generalizes Kneser's construction in two significant directions:
\begin{itemize} 
\item by explicitly incorporating isomorphisms between types, and
\item by allowing for arbitrary base schemes.
\end{itemize}
Although we retain the notation $ H(C) $ for the resulting structure, its meaning is now broader. It encompasses compositions involving non-identity isomorphisms between types and is defined over arbitrary base schemes.

In her work \cite{Wood}, Wood established a set-theoretic bijection—restricted to the primitive case—between linear binary quadratic forms and a disjoint union of quotient sets of Picard groups. Motivated primarily by moduli-theoretic considerations, she did not formulate a group law nor did she identify an explicit isomorphism with any particular Picard group. Subsequently, Dallaporta \cite{dallaporta} introduced a group structure, but again only within the context of linear binary quadratic forms. Both Wood’s and Dallaporta’s works consider \textit{linear} binary quadratic forms, whereas we are considering quadratic forms in the classical sense. While their frameworks operate entirely outside the Clifford algebra perspective, the principal innovation of our approach lies in systematically incorporating the classical Clifford-theoretic viewpoint, which we show also allows for an equally intrinsic and algebraically robust formulation of the theory. Theorem \ref{thegausscomp} gives the composition law on similarity classes of primitive binary quadratic forms taking values in line bundle over an arbitrary base scheme $S$. 

As an example of the advantages of our approach, we get the direct sum $C\oplus E$ corresponding to a Wood's pair (respectively, Clifford pairs) admits a natural quaternion algebra structure when the associated binary quadratic form $(E, q, L)$ satisfies $L\cong \OS$ or equivalently, when $C/\OS \cong \wedge^2E$.

The layout of this paper is as follows. Section~2  recalls the necessary definitions,
notations, and results, after which the main theorem is proved in Section~3.
Section~4 presents an alternative approach to the classification of primitive binary
quadratic forms, and Section~5 discusses applications of the main theorem.

\paragraph{\textbf{Acknowledgements.}}
 Soham Mondal extends his heartfelt gratitude to the University Grants Commission (UGC) for awarding him the prestigious UGC NET-JRF fellowship. This fellowship, under his student ID 423142, has been instrumental in enabling him to dedicate himself fully to this research without the burden of financial constraints. He is equally grateful to the following esteemed faculty members of the Department of Mathematics, IIT Madras, for their invaluable financial support through their research projects:
\begin{itemize}
    \item Prof. Barun Sarkar 
    \item Prof. Arijit Dey 
    \item Prof. Kunal Krishna Mukherjee 
\end{itemize}

\section{Preliminaries}
In this section, we recollect the definitions, notations, and results needed in this paper.

\begin{definition}[\textbf{Bilinear Forms and Symmetric Bilinear Forms}] \label{def:bilinear-forms}
Let $ S $ be a general scheme, $ L $ a line bundle on $ S $, and $ E $ a vector bundle on $ S $.

\begin{enumerate}
    \item An \textbf{$ L $-valued bilinear form} on $ S $ is a triple $ (E, b, L) $, where $ b : E \times E \to L $ is a bilinear $ \OS $-module morphism. The morphism $ b $ assigns to every pair of local sections $ s_1, s_2 \in \Gamma(U, E) $ over an open subset $ U \subseteq S $, a section $ b(s_1, s_2) \in \Gamma(U, L) $, satisfying the bilinearity condition in the first argument:
    \[
    b(f s_1 + g s_2, t) = f b(s_1, t) + g b(s_2, t), \quad \text{for all } f, g \in \mathscr{O}_S(U),
    \]
    and similarly in the second argument.
    \item An $ L $-valued bilinear form $ (E, b, L) $ is called \textbf{symmetric} if $ b $ is invariant under the naive switch morphism $ \sigma : E \times E \to E \times E $, defined by $ \sigma(e_1, e_2) = (e_2, e_1) $. Explicitly, $ b $ satisfies:
    \[
    b(e_1, e_2) = b(e_2, e_1),
    \]
    for all local sections $ e_1, e_2 \in \Gamma(U, E) $. Equivalently, $ b \circ \sigma = b $.
\end{enumerate}
\end{definition}
This definition generalizes the classical notion of bilinear forms and symmetric bilinear forms to the context of schemes and vector bundles, providing a framework for studying such structures in algebraic geometry.

\begin{definition}[\textbf{Quadratic Form}]\label{quad form}
    An ($L$-valued) quadratic form on $S$ is a triple $(E, q, L)$, where $E$ is a vector bundle and $q:E\rightarrow L$ is a map of sheaves satisfying the following conditions:
\begin{itemize}
    \item The following diagram of maps of sheaves is commutative,

\[\begin{tikzcd}[row sep = large, column sep = large]
 \OS \otimes E \arrow[d, "(-)^2\otimes q"'] \arrow[r, "\simeq"] & E \arrow[d, "q"] \\
 \OS \otimes L \arrow[ur, phantom, "\scalebox{1.5}{$\circlearrowleft$}" description]\arrow[r, "\simeq"] & L
 \end{tikzcd}\]
 equivalently, on sections over $U\subseteq S$, we have
 \[q(av)=a^2q(v), ~ \forall ~ a \in \Gamma(U, \OS), v\in \Gamma(U, E).\]
 \item The corresponding polar form $b_q: E \times E \rightarrow L$, defined on sections over $U\subseteq S$ by
 \[b_q(v, w)=q(v+w)-q(v)-q(w), ~ \forall ~v,w \in \Gamma(U, E),\]
 is an $L$-valued bilinear form on $S$.
 \end{itemize}
 \end{definition}

    Let $S$ be an arbitrary scheme. Fix vector bundles $E$ and $L$ on $S$. Denote by $\mathrm{Sym}_2E$, the submodule of symmetric tensor squares of $E$. We give two equivalent notions of a quadratic form on $E$ with values in the line bundle $L$.
    \bigskip
\begin{lemma}\cite[Lemma 1.1]{Auel_2014}\label{equivalent def}
The following sets of objects are in natural bijection:
\begin{enumerate}
    \item Morphisms of sheaves \( q : E \to L \) satisfying:
    \[
    q(av) = a^2 q(v),
    \]
    for sections \( a \) of \( \OS \) and \( v \) of \( E \); and such that the morphism of sheaves 
    \[
 b_q : E \times E \to L, 
    \]
defined by
    \[
    b_q(v, w) = q(v + w) - q(v) - q(w),
    \]
    for sections \( v \) and \( w \) of \( E \), is \( \OS \)-bilinear.

    \item Morphisms of \( \OS \)-modules \( \mathrm{Sym}_2E \to L \).

%\item Global sections \( \Gamma(S, Sym^2(E^\vee) \otimes L) \).
\end{enumerate}
\end{lemma}
\vspace{.05in}
\begin{definition}[\textbf{Isometry of Quadratic Forms}]
    An isometry of ($L$-valued) quadratic forms $\varphi : (E, q, L) \simeq (E', q', L)$ is an $\OS$-module isomorphism $\varphi :E \rightarrow E'$ such that the following diagram of maps of $\OS$-modules commutes:

\[\begin{tikzcd}[row sep = large, column sep = large]
  E \arrow[d, "\varphi"] \arrow[r, "q"] & L \arrow[d, "Id"] \\
 E' \arrow[ur, phantom, "\scalebox{1.5}{$\circlearrowleft$}" description]\arrow[r, "q'"] & L
 \end{tikzcd}\]
 
 \end{definition}
    \vspace{.09in}
\begin{definition}[\textbf{Similarity}] 
A similarity or similitude between two quadratic forms $(E, q, L)$ and $(E', q', L')$ is a pair $(\varphi, \mu_\varphi)$ of $\OS$-module isomorphisms $\varphi :E \rightarrow E'$ and $\mu_\varphi :L \rightarrow L'$ such that the following diagram of $\OS$-modules commutes:
\[\begin{tikzcd}[row sep = large, column sep = large]
  E \arrow[d, "\varphi"] \arrow[r, "q"] & L \arrow[d, "\mu_\varphi"] \\
 E' \arrow[ur, phantom, "\scalebox{1.5}{$\circlearrowleft$}" description]\arrow[r, "q'"] & L'
 \end{tikzcd}\]
 
\end{definition}

\vspace{.05in}

\begin{definition}[\textbf{Discriminant of a Quadratic Form}]
Let $ S $ be a scheme, let $ E $ be a locally free $ \OS $-module of rank $ n $, and let $ L $ be a line bundle on $ S $. Let $q$ be a  \emph{quadratic form} on $ E $ with values in $ L $, namely a morphism of sheaves $ q \colon E \to L $ such that  
$$
q(av) = a^2 q(v)
$$  
for all local sections $ a \in \OS $, $ v \in E $, together with an associated symmetric bilinear form $ b_q \colon E \times E \to L $.

Locally, on an affine Zariski open cover $ \{U_i\} $ of $S$ where $ E|_{U_i} \cong \mathscr{O}_{U_i}^{\oplus n} $ and $ L|_{U_i} \cong \mathscr{O}_{U_i} $, the quadratic form is given by  
$$
q\left( \sum x_k e^{(i)}_k \right) = \sum_{k \le l} b^{(i)}_{kl} x_k x_l,
$$  
with $ b^{(i)}_{kl} \in \OS(U_i) $, and $ b_q $ has matrix $ (B^{(i)}_{kl}) $, where $ B^{(i)}_{kk} = 2b^{(i)}_{kk} $ and $ B^{(i)}_{kl} = b^{(i)}_{kl} $ for $ k \ne l $.

Then the \emph{discriminant} of $ q $ is the global section  
$$
\Delta(q) \in H^0(S, L^{\otimes n} \otimes (\det E)^{\otimes -2})
$$  
defined locally by $ \Delta_i = \det(B^{(i)}_{kl}) \in \OS(U_i) $.  For  details, see \cite{sohamthesis}.
\end{definition}
\vspace{.05in}
\begin{definition}[\textbf{Quadratic Algebra}]
A quadratic $\OS$-algebra is a sheaf $C$ of $\OS$-algebras that is locally free of rank 2 as a sheaf of $\OS$-modules i.e., there is a basis of open sets $U$ of $S$ such that $C(U)$ is a free $\OS(U)$-module.
\end{definition}

\vspace{.05in}
\begin{example}
    If \( L/K \) is a quadratic field extension, then the structure sheaf of \( \operatorname{Spec}(L) \) is a quadratic algebra over \( \operatorname{Spec}(K) \).

\end{example}
\vspace{.05in}

\begin{lemma}\cite[Proposition 2.4]{dallaporta}\label{freedirectsummand}
Let $C$ be a quadratic algebra over $S$. Then $C/ \OS$ is a locally free $\OS$-module of rank $1$. Moreover, if $C$ is free as an $\OS$-module, then there exists ${\tau \in \Gamma(S,C)}$ such that ${C= \OS \oplus \tau \OS}$ as $\OS$-modules. In particular, ${C/\OS = \tau \OS}$ is free.
\end{lemma}
\vspace{.05in}

\begin{rmk}[Failure of Global Splitting for Sheaves over Schemes]
\label{remark:failure_of_splitting}
If $ R $ is a commutative ring and $ C $ is a quadratic $ R $-algebra, it is a classical result that the module $ C $ decomposes as
$$
C \cong R \oplus (C/R)
$$
as $ R $-modules; see \cite[Lemma 1.3.5]{knus} and \cite[Lemma 1.3]{Lowrank}. This decomposition reflects the fact that, in the affine setting, the canonical inclusion $ R \hookrightarrow C $ admits a retraction, thereby ensuring that the quotient map $ C \twoheadrightarrow C/R $ splits.

However, this result does not generalize to the setting of schemes, as illustrated by the examples in \cite{Bhatt_2012} and \cite[Example~2.10]{sohamthesis}. Nevertheless, global splitting holds when the base scheme $ S $ satisfies $ \frac{1}{2} \in \Gamma(S, \mathscr{O}_S) $.

\end{rmk}
\vspace{.05in}
\begin{lemma}\cite[Lemma 2.11]{sohamthesis} \label{directsummand}
 Let $S$ be a scheme such that $\frac{1}{2}\in \Gamma (S, \OS)$ then $\OS \cdot 1_C \hookrightarrow C$ is an $\OS$-direct summand of $C$.
 \end{lemma}

    \vspace{.05in}

\begin{definition}[\textbf{Discriminant of a Quadratic Algebra}]\label{discalg}
The \textbf{discriminant} of a quadratic algebra $C$ is defined to be the pair:
$$
(\Delta(C), N)
$$
where:
\begin{itemize}
    \item $N = (C/\OS)^\vee$, a line bundle (invertible sheaf) on $S$,
    \item $\Delta(C) \in H^0(S, N^{\otimes 2})$ is a global section of $N^{\otimes 2}$, constructed locally as follows.
\end{itemize}
\end{definition}
\paragraph*{{\textnormal{\textit{Local Construction.}}}}
Let $U \subseteq S$ be an open subset where $C|_U$ is free of rank 2 as an $\OU$-module. Choose a basis $\{1, \tau\}$, so that:
$$
\mathcal{C}|_U \simeq \OU[\tau]/(\tau^2 + r\tau + s), \quad \text{with } r,s \in \Gamma(U, \OU)
$$

Then:
\begin{itemize}
    \item  Let $\overline{\tau} \in (C/\OS)|_U$ be the image of $\tau$,
    \item Let $\overline{\tau}^\vee \in N(U)$ be the dual basis element corresponding to $\tau$,
    \item  Then
    $$
    \Delta(C)|_U := (r^2 - 4s)(\overline{\tau}^\vee \otimes \overline{\tau}^\vee) \in N^{\otimes 2}(U)
    $$

\end{itemize}

Let $ S $ be a scheme and let $ C $ be a quadratic algebra over $ S $, i.e., a sheaf of commutative $ \OS $-algebras which is locally free of rank 2.  The local definition of the discriminant :
$$
\Delta(C)|_U := (r^2 - 4s)(\overline{\tau}^\vee \otimes \overline{\tau}^\vee) 
$$
glues across overlaps, i.e., it defines a well-defined global section of $ ((C/\OS)^\vee)^{\otimes 2} $. We say that two discriminants \((\Delta, N)\) and \((\Delta', N')\) are \emph{isomorphic} if there exists an isomorphism of \(\mathscr{O}_S\)-modules
\[
f : N \xrightarrow{\cong} N'
\]
such that
\[
\Delta' = f^{\otimes 2}(\Delta),
\]
where \(f^{\otimes 2} : N^{\otimes 2} \xrightarrow{\cong} (N')^{\otimes 2}\) is the induced isomorphism on the tensor squares.
For a more detailed discussion, see \cite{sohamthesis}.

\vspace{.05in}
\begin{definition}[\textbf{Generalized Even Clifford algebra}]
\upshape
    In his dissertation, Bichsel \cite{bichsel1985quadratische} introduced the construction of the even Clifford algebra associated with a quadratic form taking values in a line bundle over an affine scheme. Subsequent alternative approaches were developed by Bichsel–Knus \cite{bichselknus}, Caenepeel–van Oystaeyen \cite{caenepeel}, and Parimala–Sridharan \cite{parimalareduced}. We now recall the direct tensorial construction as presented by Asher Auel \cite{auel2015surjectivity}.  
    
    Let $(E, q, L)$ be a (line bundle-valued) quadratic form of rank $n$ over a scheme $S$, where $n = 2m$ or $n = 2m + 1$. Consider the tensor algebra $T(E \otimes E \otimes L^\vee)$, and denote by $J_1^{C_0}$ and $J_2^{C_0}$ the ideals generated respectively by elements of the form  
$$
v \otimes v \otimes f - f(q(v)) \cdot 1,
$$
and  
$$
u \otimes v \otimes f \otimes v \otimes w \otimes g - f(q(v)) \, u \otimes w \otimes g,
$$
for local sections $u, v, w$ of $E$ and $f, g$ of $L^\vee$.

In Asher Auel’s work \cite{auel2015surjectivity},  the generalized even Clifford algebra of $(E, q, L)$ is introduced as the quotient  
$$
C_0(E, q, L) = T(E \otimes E \otimes L^\vee) / (J_1^{C_0} + J_2^{C_0}).
$$

There exists a canonical $\OS$-module morphism  
$$
i_{C_0} \colon E \otimes E \otimes L^\vee \to C_0(E, q, L),
$$  
obtained by composing the natural inclusion  
$$
\theta \colon E \otimes E \otimes L^\vee \hookrightarrow T(E \otimes E \otimes L^\vee)
$$  
with the projection  
$$
p \colon T(E \otimes E \otimes L^\vee) \twoheadrightarrow C_0(E, q, L),
$$  
factoring through the degree-one component of the tensor algebra.
\end{definition}

\vspace{.05in}
\begin{proposition}[Universal Property of the generalized even Clifford algebra]\cite[Proposition 1.1]{auel2015surjectivity}\label{universalC_0} 
Given an $\OS$-algebra $A$ and an $\OS$-module morphism $j:E\otimes E \otimes L^\vee \rightarrow A$ such that
$$
j(v\tensor v \tensor f) = f(q(v))\cdot 1 \quad \mbox{and} \quad 
j(u\tensor v \tensor f) \cdot j(v \tensor w \tensor g) = 
f(q(v)) \, j(u \tensor w \tensor g),
$$
then there exists a unique $\OS$-algebra homomorphism $\psi :
C_0(E, q, L) \to A$ satisfying $j = \psi \circ i_{C_0}$.  
\end{proposition}
\begin{proof}
    See \cite[Proposition 2.19]{sohamthesis} for details.
\end{proof}
\vspace{.05in}
\begin{rmk}
From this point onward, by the \emph{even Clifford algebra} of a quadratic form $(E, q, L)$, where $L$ is a non-trivial line bundle, we shall refer to its \emph{generalized even Clifford algebra}. This extension accounts for the twist by $L$. 
\end{rmk}
\vspace{.05in}

    \begin{proposition}\label{induceevenclifford}
    Let \( (E, q, L) \) be a  quadratic form of rank \( n \) on a scheme \( S \). Write \( n = 2m \) or \( n = 2m + 1 \).

\begin{itemize}
\item[(a)] \( C_0(E, q, L) \) is a locally free \( \mathscr{O}_S \)-algebra of rank \( 2^{n-1} \).

 \item[(b)] Any similarity \( (\varphi, \lambda): (E, q, L) \to (E', q', L') \) induces an \( \mathscr{O}_S \)-algebra isomorphism
  \[
  C_0(\varphi, \lambda): C_0(E, q, L) \xrightarrow{\cong} C_0(E', q', L').
  \]
  
\item[(c)] For any morphism of schemes \( p: S' \to S \), there is a canonical \( \mathscr{O}_{S'} \)-module isomorphism
  \[
  C_0(p^*(E, q, L)) \xrightarrow{\cong} p^* C_0(E, q, L).
  \]
\end{itemize}
\end{proposition}

\begin{proof}
The result is a consequence of \cite[Proposition 1.2]{auel2015surjectivity}. See also \cite[Proposition 2.20 and Proposition 2.21]{sohamthesis}. 
\end{proof}
\vspace{.05in}

\begin{proposition} \label{quotientclifford}
Let $ q: E \to L$ be a binary quadratic form. Then the quotient $ C_0(E, q, L)/\OS $ is canonically isomorphic to $ (\wedge^2 E \otimes L^\vee) $ as $\OS$-modules.
\end{proposition}

\begin{proof}

The even Clifford algebra $ C_0(E, q, L) $ satisfies a universal property with respect to the quadratic form $ q $. Specifically, there exists a natural map induced by the tensor product structure:
$$
E \otimes E \otimes L^\vee \xrightarrow{\theta} T(E \otimes E \otimes L^\vee) \xrightarrow{p} C_0(E, q, L),
$$
where:
\begin{itemize}
    \item $ T(E \otimes E \otimes L^\vee) $ denotes the tensor algebra of $ E \otimes E \otimes L^\vee $,
    \item $ p $ is the quotient map onto the even Clifford algebra $ C_0(E, q, L) $,
    \item $ \theta $ is the canonical inclusion of $ E \otimes E \otimes L^\vee $ into the tensor algebra.
\end{itemize}

Quotienting further by the action of $ \OS $, we obtain the following sequence of maps:
$$
E \otimes E \otimes L^\vee \xrightarrow{\theta} T(E \otimes E \otimes L^\vee) \xrightarrow{p} C_0(E, q, L) \xrightarrow{Q} C_0(E, q, L)/\OS.
$$

On the level of sections over an open subset $ U \subset S $, the composition of these maps vanishes on elements of the form:
$$
\{ x \otimes x \otimes f : x \in \Gamma(U, E), f \in \Gamma(U, L^\vee) \}.
$$
This follows directly from the defining relations of the Clifford algebra.
By the universal property of quotient, there exists a unique morphism of $\OS$-modules:
$$
\xi: \wedge^2 E \otimes L^\vee \to C_0(E, q, L)/\OS.
$$

To establish surjectivity of $ \xi $, we localize the problem. Now we can assume $E$ and $L$ are free $R$-modules. By the Poincaré–Birkhoff–Witt (PBW) theorem \cite[Chapter IV, Theorem 1.5.1, \S1]{knus} and the binary nature of the quadratic form $ q $ ensures that $ C_0(E, q, L) $ is generated locally by elements of the form $1$ and $( x \wedge y)  $, where $ x, y \in E $. Thus, the map $ \xi $ is surjective after localizing and hence globally surjective.

Both $ C_0(E, q, L)/\OS $ and $ \wedge^2 E \otimes L^\vee $ are locally free $\OS$-modules of rank 1. This follows from the fact that $ E $ is a rank-2 vector bundle (since $ q $ is binary), and hence $ \wedge^2 E $ is a line bundle. Tensoring with $ L^\vee $, another line bundle, preserves the rank.

Since $ \xi $ is a surjective morphism between two locally free $\OS$-modules of the same rank, it must also be injective. Therefore, $ \xi $ is an isomorphism.
\end{proof}
\vspace{.05in}
\begin{rmk}\label{localsplit}
In view of Remark~\ref{remark:failure_of_splitting}, we may locally express the sheaf $ C_0(E, q, L) $ as
$$
C_0(E, q, L) \cong \OS \oplus (\wedge^2 E \otimes L^\vee).
$$
\end{rmk}

\vspace{.05in}
\begin{lemma}[Discriminant Comparison for Binary Quadratic Forms and Their Associated even Clifford Algebras]\cite[Lemma 2.26]{sohamthesis}\label{discriminant-comparison}
Let $ S $ be a general scheme, and let $ q : E \to L $ be a binary quadratic form over $ S $, where $ E $ is a rank-$ 2 $ vector bundle and $ L $ is a line bundle on $S$. Denote by $ C_0(E, q, L) $ the even part of the associated generalized Clifford algebra. Then the discriminant of the quadratic form $ q $ and the discriminant of the quadratic algebra $ C_0(E, q, L) $ differ only by a sign. 
  \end{lemma}
\begin{proof}
    For a detailed proof, see \cite[Lemma 2.28]{sohamthesis}.
\end{proof}

\begin{definition}[\textbf{Clifford bimodule}]
\upshape
    Here, we recall its definition from \cite{auel2015surjectivity}, emphasizing its functorial nature and compatibility with base change.

The \( \OS \)-module \( E \otimes T(E \otimes E \otimes L^\vee) \) naturally admits a right \( T(E \otimes E \otimes L^\vee) \)-module structure, denoted here by \( \otimes \). An \( \OS\)-bilinear map 
\[
\ast : (E \otimes E \otimes L^\vee) \times E \to E \otimes (E \otimes E \otimes L^\vee)
\]
defined by
\[
(u \otimes v \otimes f) \ast w = u \otimes (v \otimes w \otimes f),
\]
for sections \( u, v, w \) of \( E \) and \( f \) of \( L^\vee \), induces a left \( T(E \otimes E \otimes L^\vee) \)-module structure \( \ast \) on \( E \otimes T(E \otimes E \otimes L^\vee) \). This structure is uniquely determined by its commutativity with the right \( T(E \otimes E \otimes L^\vee) \)-module structure. Consequently, we define
\[
C_1(E, q, L) = \frac{E \otimes T(E \otimes E \otimes L^\vee)}{E \otimes J_1^{C_0} + J_1^{C_0} \ast E},
\]
accompanied by the canonically induced \( \OS \)-module morphism
\[
i: E \to C_1(E, q, L),
\]
which serves as a locally split embedding. It can be readily verified that \( E \otimes J_2^{C_0} \subset J_1^{C_0} \ast E \) and \( J_2^{C_0} \ast E \subset E \otimes J_1^{C_0} \), ensuring that \( C_1(E, q, L) \) inherits a bimodule structure over \( C_0(E, q, L) \). Denote the right and left \( C_0(E, q, L) \)-module structures by \( \cdot \) and \( \ast \), respectively.

Assuming \( \text{rank}(E) = n = 2m \) or \( n = 2m + 1 \), the module admits a filtration
\[
E = F_1 \subset F_3 \subset \cdots \subset F_{2m+1} = C_1(E, q, L),
\]
where \( F_{2i+1} \) denotes the image of the truncation \( E \otimes T^{\leq i}(E \otimes E \otimes L^\vee) \) in \( C_1(E, q, L) \) for \( 0 \leq i \leq m \). The associated graded pieces satisfy
\[
F_{2i+1}/F_{2i-1} \cong \wedge^{2i+1} E \otimes (L^\vee)^{\otimes i}.
\]
Thus, \( C_1(E, q, L) \) is a locally free \( \OS \)-module of rank \( 2^{n-1} \).

\end{definition}
\vspace{.05in}
\begin{proposition}[Universal Property of the Clifford Bimodule]\cite[Proposition 1.4] {auel2015surjectivity}\label{universalbimodule}
Let $ B $ be a $ C_0(E, q, L) $-bimodule, with right and left actions denoted by $ \cdot $ and $ \ast $, respectively. Let
$ j : E \to B $ be an $ \OS $-module morphism satisfying the identity
$$
j(u) \cdot i_{C_0}(v \otimes w \otimes f) = i_{C_0}(u \otimes v \otimes f) \ast j(w),
$$
for all sections $ u, v, w \in \Gamma(U, E) $ and $ f \in \Gamma(U, L^\vee) $, where $ U \subseteq S $ is any open subset. Then there exists a unique $ C_0(E, q, L) $-bimodule morphism
$$
\psi : C_1(E, q, L) \to B
$$
such that $ j = \psi \circ i $.
\end{proposition}
\begin{proof}
    See \cite[Proposition 2.28]{sohamthesis} for a detailed proof.
\end{proof}

\vspace{.05in}

\begin{proposition} \cite[Proposition 1.5]{auel2015surjectivity}\label{semilinear}
    Let \( (E, q, L) \) be a  quadratic form of rank \( n \) on a scheme \( S \). Write \( n = 2m \) or \( n = 2m + 1 \).

\begin{itemize}
\item[(a)] There is a canonical isomorphism of $ \mathscr{O}_{S'} $-modules:
$$
C_1(p^*(E, q, L)) \xrightarrow{\cong} p^* C_1(E, q, L).
$$

 \item[(b)]  Any similarity transformation \( (\varphi, \lambda) : (E, q, L) \to (E', q', L') \) induces an \( \OS \)-module isomorphism  
\[
C_1(\varphi, \lambda) : C_1(E, q, L) \to C_1(E', q', L'),
\]  
which is \( C_0(\varphi, \lambda) \)-semilinear, as illustrated by the following commutative diagram:  
 \[ \begin{tikzcd}
C_1(E, q, L) \otimes_{C_0(E, q, L)} C_1(E, q, L)  \arrow{rr}{C_1(\varphi, \lambda) \otimes C_1(\varphi, \lambda)} \arrow{d}{m} & & C_1(E', q', L') \otimes_{C_0(E', q', L')} C_1(E', q', L') \arrow{d}{m} \\%
C_0(E, q, L) \otimes_{\OS} L \arrow{rr}{C_0(\varphi, \lambda) \otimes \lambda}&& C_0(E', q', L') \otimes_{\OS} L'
\end{tikzcd}
\]
where m represents the multiplication maps in the respective modules.
 \end{itemize}
\end{proposition}
\vspace{.05in}
\begin{theorem}\label{cliffordbimoduleunderlyingmodule}
    The Clifford bimodule \( C_1(E, q, L) \) for a binary quadratic form (E, q, L) is isomorphic to the underlying rank 2 vector bundle \( E \).
\end{theorem}

\begin{proof}
The assertion follows from the universal property of the Clifford bimodule (Proposition~\ref{universalbimodule}); for details, see \cite[Remark~2.32]{sohamthesis}.
\end{proof}

\vspace{.05in}
\begin{rmk}
    From now on, via the canonical identification, $ E $ will be identified with $ C_1(E, q, L) $, in the setting of the binary quadratic form $ (E, q, L) $.
\end{rmk}

We next recall Wood's notation of a \textit{traceable} module and some of its properties.
\vspace{.05in}
\begin{definition}[\textbf{Traceable Module}]\cite[Definition 1.2]{traceablewood}
Let $ f : X \to T $ be a double cover and consider an $ \mathscr{O}_X $-module $ \mathcal{M} $. We say that $ \mathcal{M} $ is \textit{traceable}  if:
\begin{enumerate}
    \item The push-forward $ f_*\mathcal{M} $ is a locally free $\mathscr{O}_T $-module of rank 2, and
    \item $ f_*\mathcal{M} $ and $ f_*\mathscr{O}_X $ give the same trace map $ f_*\mathscr{O}_X \to\mathscr{O}_T $, i.e., the composite maps
    $$
    \begin{tikzcd}[column sep=large]
    f_*\mathscr{O}_X \arrow[r, "mult."] & \underline{\mathrm{End}}_{\mathcal{O}_T}(f_*\mathscr{O}_X, f_*\mathscr{O}_X) \arrow[r, "\mathrm{trace}"] &\mathscr{O}_T
    \end{tikzcd}
    $$
    and
    $$
    \begin{tikzcd}[column sep=large]
    f_*\mathscr{O}_X \arrow[r, "mult."] & \underline{\mathrm{End}}_{\mathcal{O}_T}(f_*\mathcal{M}, f_*\mathcal{M}) \arrow[r, "\mathrm{trace}"] &\mathscr{O}_T
    \end{tikzcd}
    $$
    agree.
\end{enumerate}
\end{definition}
 A local description of \textit{traceable} modules in terms of commutative algebra is as follows.
Assume that $ T = \Spec(A) $, where $ A =\Gamma(T, \mathscr{O}_T)$ is a commutative ring, $ X = \Spec(B) $, where $ B=\Gamma(X, \mathscr{O}_X) $ is a free $ A $-module of rank 2 and $ M $ is a finitely generated $ B $-module such that $ f_*M = M $ is a free $ A $-module of rank $2$.
The multiplication map on $ B $ is:
$$
B \otimes_A B \to B,\quad b_1 \otimes b_2 \mapsto b_1b_2,
$$
Which induces a map:
$$
\mu_B: B \to \End_A(B)
$$
sending $ b \in B $ to the endomorphism $ x \mapsto bx $. Similarly, for $ M $, we get:
$$
\mu_M: B \to \End_A(M)
$$
sending $ b \in B $ to $ m \mapsto bm $.

Since both $ B $ and $ M $ are free $ A $-modules of rank 2, the endomorphism rings $ \End_A(B) $ and $ \End_A(M) $ have well-defined trace maps to $ A $, defined as the trace of the linear operator acting on the module.

Thus we get two composite maps:

    \[
    \text{Tr}_{\mathscr{O}_X}: B \xrightarrow{\mu_B} \End_A(B) \xrightarrow{\tr} A \quad \text{for}~~ \mathscr{O}_X~~and
    \]

    \[
    \text{Tr}_M: B \xrightarrow{\mu_M} \End_A(M) \xrightarrow{\tr} A \quad for ~~\mathcal{M}.
    \]

The traceability of $\mathcal{M}$ then requires that these two maps agree:
$$
\text{Tr}_{\mathscr{O}_X}(b) = \text{Tr}_{\mathcal{M}}(b) \quad \text{for all } b \in B.
$$

That is, for every $ b \in B $, the trace of the operator $ x \mapsto bx $ on $ B $ equals the trace of the operator $ m \mapsto bm $ on $ M $.

\vspace{.05in}
\begin{lemma}\label{traceablebimodule}
    If $q:E\rightarrow L$ be a binary quadratic form, then the Clifford bimodule, $C_1(E, q, L),$ is a \textit{traceable} $C_0(E, q, L)$-module.
\end{lemma}
\begin{proof}
    By Theorem~\ref{cliffordbimoduleunderlyingmodule}, we may assume  without loss of generality  that the Clifford bimodule corresponding to the binary quadratic form $(E, q, L)$ coincides with the original vector bundle, i.e.,
\[
C_1(E, q, L) = E.
\]
We aim to show that $ C=C_0(E, q, L) $ and $ E $ induce the same trace map $ C \to \OS $. To prove equality of two maps between sheaves of $ \OS $-modules, it suffices to check their agreement locally on open subsets of $ S $. Locally, we can assume $ C $, $ E $, and $ L $ are free modules over $ R $, simplifying computations. Let $ E $ be the rank-2 free $ R $-module with basis $ \{e_1, e_2\} $. The quadratic form $ q $ is given by:
\[
q(xe_1 + ye_2) = ax^2 + bxy + cy^2,
\]
where $ a, b, c \in R $ are coefficients defining the quadratic form. The even Clifford algebra $ C_0(E, q, L) $ is a free $ R $-module with basis $ \{1, \tau\} $, where $ \tau = e_1 \cdot e_2 $. The multiplication in $ C $ satisfies the following relations:
\begin{align*}
    \tau^2 &= b\tau - ac, \\
    \tau e_1 &= be_1 - ae_2, \\
    \tau e_2 &= ce_1.
\end{align*}
 Recall that the trace map is $ R $-linear, so it suffices to examine the action of left multiplication by the generators of $ C $. The generators of $ C $ are $ 1 $ and $ \tau $. Since $ 1 $ acts trivially (multiplication by $ 1 $ does not change the module), we focus on $ \tau $. In $ C $, the basis is $ \{1, \tau\} $. The action of $ \tau $ is given by:
\[
\tau \cdot 1 = \tau, \quad \tau \cdot \tau = b\tau - ac.
\]
This corresponds to the matrix:
\[
\begin{pmatrix}
0 & -ac \\
1 & b
\end{pmatrix}.
\]
In $ E $, the basis is $ \{e_1, e_2\} $. Using the relations $ \tau e_1 = be_1 - ae_2 $ and $ \tau e_2 = ce_1 $, the action of $ \tau $ is represented by:
\[
\begin{pmatrix}
b & c \\
-a & 0
\end{pmatrix}.
\]

The trace of a matrix is the sum of its diagonal entries. For both matrices, since the traces are equal, the trace maps induced by $ C $ and $ E $ coincide.

\end{proof}

\subsection{Standard Involution and Quaternion algebra}
The following definitions are drawn from \cite{knus}, originally formulated in the context of rings; however, they admit a natural generalization to the setting of an arbitrary base scheme $S$. 

\begin{definition}[\textbf{Involution}]
Let \( R \) be a ring. An \emph{involution} on \( R \) is a map  
\[
\sigma: R \to R
\]
satisfying the following properties for all \( a, b \in R \):  
\begin{enumerate}
    \item \textit{Anti-homomorphism:}  
    \[
    \sigma(ab) = \sigma(b) \sigma(a).
    \]
    \item \textit{Additivity:}  
    \[
    \sigma(a + b) = \sigma(a) + \sigma(b).
    \]
    \item \textit{Involutivity:}  
    \[
    \sigma^2(a) = a.
    \]
\end{enumerate}

If \( R \) is a ring with unity \( 1 \), then an involution must also satisfy \( \sigma(1) = 1 \). The pair \( (R, \sigma) \) is called a \emph{ring with involution}.  
\end{definition}

Let \( R \) be a commutative ring with involution \( \sigma_R \).  
An \emph{\( R \)-algebra with involution} is an \( R \)-algebra \( A \) with an involution \( \sigma_A \) that extends the involution of \( R \):  
\[
\sigma_A(ra) = \sigma_R(r) \sigma_A(a) \quad \text{for } a \in A, \, r \in R.
\]

Let \( (A, \sigma_A) \) and \( (B, \sigma_B) \) be rings with involution.  
A \emph{morphism of rings with involution}  
\[
\varphi \colon (A, \sigma_A) \to (B, \sigma_B)
\]  
is a homomorphism of rings \( \varphi \colon A \to B \) such that  
\[
\sigma_B (\varphi(a)) = \varphi (\sigma_A(a)), \quad a \in A.
\]  
    \vspace{.05in}
\begin{definition}[\textbf{Standard Involutions}]
Let \( R \) be a commutative ring. For any \( R \)-algebra \( A \) with involution \( \sigma \), we define the \textit{trace} \(\operatorname{tr}\) and the \textit{norm} \( n \) with respect to the involution as
\[
\operatorname{tr} (x) = x + \sigma(x), \quad n(x) = x\sigma(x), \quad x \in A.
\]
We say that \( \sigma \) is a \textit{standard involution} if  

\begin{itemize}
    \item  \( R \) is fixed under \( \sigma \), i.e., \( \operatorname{Fix}(\sigma) \supseteq R \).  
    \item \( \operatorname{tr} (x) \in R \) and \( n(x) \in R \) for all \( x \in A \). 
\end{itemize}
\end{definition}

\vspace{.05in}

\begin{proposition}\cite[Proposition 1.3.4]{knus}
    Let $A$ be an $R$-algebra which is a finitely generated projective and faithful $R$-module. Then there exists at most one standard involution on $A$.
\end{proposition}
\vspace{.05in}

\begin{rmk}
    Owing to the uniqueness of standard involutions, their definition naturally extends to locally free sheaves of algebras over general schemes. This extension follows from the ability to define the involution locally and then uniquely glue the local definitions together. To formalize this construction, we proceed as follows:

\begin{enumerate}
    \item \textit{Local Data}: For each affine open subset $ U \subset X $, define $ \sigma_U : \mathcal{A}(U) \to \mathcal{A}(U) $ as the unique standard involution on $ \mathcal{A}(U) $.

    \item \textit{Compatibility on Overlaps}: Suppose $ U $ and $ V $ are affine open subsets of $ X $ with $ U \cap V \neq \emptyset $. The restrictions of $ \sigma_U $ and $ \sigma_V $ to $ U \cap V $ must coincide because the standard involution is unique. Thus:
    \[
        \sigma_U|_{U \cap V} = \sigma_V|_{U \cap V}.
    \]

    \item \textit{Gluing}: By the sheaf property of $ \mathcal{A} $, the local maps $ \sigma_U $ glue together to define a global map $ \sigma : \mathcal{A} \to \mathcal{A} $.

    \item \textit{Properties of the Global Involution}: The global map $ \sigma $ inherits the properties of the local involutions:
    \begin{itemize}
        \item $ \sigma^2 = \text{id} $,
        \item $ \sigma $ is $ \OS $-linear,
        \item $ \sigma $ satisfies any additional structural properties (e.g., anti-multiplicativity).
    \end{itemize}
\end{enumerate}
\end{rmk}
\vspace{.05in}
\begin{definition}[\textbf{Quaternion Algebra}]
 We say that an $R$-algebra $A$ is a \textit{quaternion algebra} if
\begin{enumerate}
    \item $A$ is a projective $R$-module of rank 4,
    \item $A$ has a standard involution.
\end{enumerate}
\end{definition}
\vspace{.05in}
\begin{definition}[\textbf{Quaternion Algebra over a scheme $S$}]
Let $ S $ be a general scheme. A quaternion algebra over $ S $ is a sheaf of $ \OS $-algebras $ \mathcal{A} $ satisfying the following properties:

\begin{enumerate}
    \item \textit{Locally Free of Rank 4:} 
    The sheaf $ \mathcal{A} $ is a locally free coherent $ \OS $-module of rank 4, meaning that for every point $ s \in S $, there exists an open neighborhood $ U \subseteq S $ such that:
    $$
    \mathcal{A}|_U \simeq \mathscr{O}_U^{\oplus 4}
    $$
    as $ \mathscr{O}_U $-modules.

    \item \textit{Existence of a Standard Involution:}
    There exists a global $ \OS $-linear involution:
    $$
    \sigma: \mathcal{A} \to \mathcal{A}
    $$
    called the \emph{standard involution}, satisfying the following properties for all local sections $ a \in \mathcal{A}(U) $:
    $$
    \sigma(a)a = a\sigma(a) \in \OS(U) \cdot 1_{\mathcal{A}}
    $$
    $$
    a + \sigma(a) \in \OS(U) \cdot 1_{\mathcal{A}}
    $$

    This allows us to define:
    
     \begin{itemize}

\item The \emph{reduced trace}: $ \mathrm{tr}(a) = a + \sigma(a) \in \OS(U) $
        
\item The \emph{reduced norm}: $ \mathrm{n}(a) = a \cdot \sigma(a) \in \OS(U) $
    
\end{itemize}  
    Each element $ a \in \mathcal{A}(U) $ satisfies its characteristic polynomial:
    $$
    a^2 - \mathrm{tr}(a) \cdot a + \mathrm{n}(a) \cdot 1 = 0
    $$
\item \textit{Local Structure:} 
    Locally on $ S $, $ \mathcal{A} $ admits a presentation analogous to the classical quaternion algebras over rings. That is, for any affine open subset $ U = \mathrm{Spec}(R) \subseteq S $, we have:
    $$
    \mathcal{A}(U) \simeq R\langle i, j \rangle / (i^2 - a,\, j^2 - b,\, ij + ji)
    $$
    for some $ a, b \in R^\times $. This is denoted by $ (a, b)_R $, and is called a \emph{symbol algebra} or \emph{quaternion algebra} over $ R $.
\end{enumerate}
    
\end{definition}

\begin{example}\label{quaternion}
    If $ q \colon E \to \mathscr{O}_S $ is a binary quadratic form over a general scheme $ S $, then the full Clifford algebra $ C(E, q, \mathscr{O}_S) $ naturally carries the structure of a quaternion algebra over $ S $.
\end{example}
\vspace{.05in}
\begin{rmk}
    Venkata Balaji established in his work \cite{tevbal} that every quaternion algebra can be realized as the even Clifford algebra of a ternary quadratic form. Specifically, for any quaternion algebra \( \mathcal{A} \) over an arbitrary scheme \( S \), there exists a rank 3 vector bundle \( V \) on \( S \), a quadratic form \( q \) on \( V \) with values in the line bundle \( L := \det^{-1}(\mathcal{A}) \), and an isomorphism of algebra bundles $\mathcal{A} \cong C_0(V, q, L)$, where \( C_0(V, q, L) \) denotes the even Clifford algebra associated with the triple \( (V, q, L) \). See also \cite{voightcharacterizing} and \cite{chanconic}.

\end{rmk}

\section{Classification of binary quadratic forms using Clifford pairs}
In this section we prove the main result of this work.

\begin{theorem}\label{MainTheorem}
    For any scheme $S$
, the natural map
    \[(E, q, L)\mapsto (C_0(E, q, L), C_1(E, q, L))\]
    induces a discriminant-preserving (up to sign); specifically, satisfying $$
\Delta(q) = -\Delta(C_0(E, q, L)),
$$ bijective correspondence
\[
\left\{
\parbox{2.5in}{\centering Similarity classes of binary quadratic forms (E, q, L) over \( S \)}
\right\}
\longleftrightarrow
\left\{
\parbox{2.5in}{\centering Isomorphism classes of pairs \( (C, E) \), \\ with \( C \) a quadratic algebra over \( S \), \\ and \( E \) a \textit{traceable} \( C \)-module}
\right\},
\]
which is functorial in $S$.

An isomorphism of pairs $(C , E)$ and $(C', E')$ is given by an isomorphism $C \cong C'$ of $\OS$-algebras, and an isomorphism $E \cong E'$ of $\OS$-modules that respects the $C$ and $C'$ module structures. Here for any binary quadratic form $ q:E \rightarrow L $, the objects $ C_0(E, q, L) $ and $ C_1(E, q, L) $ denote, respectively, the 
degree zero and degree one components of the generalized Clifford algebra associated to the binary quadratic form $ q $.
\end{theorem}

\begin{proof}

 For the sake of clarity and rigor, we present the proof as a sequence of interdependent stages.
\subsection*{\textnormal{\textit{Verifying the Well-Definedness of the Map in Theorem \ref{MainTheorem}}}}
% \paragraph{\textbf{Step 1:}} \textbf{Verifying the Well-Definedness of the Map in Theorem \ref{MainTheorem}.}\\
Consider a binary quadratic form $ q \in \Gamma(S, \mathscr{Q}uad(E, L))$. Let $ C = C_0(E, q, L) $ denote the even Clifford algebra associated with $ q $. It is a classical result that $ C $ carries the structure of a quadratic algebra, and $ E $ naturally inherits the structure of a $ C $-module. Furthermore, it follows from Lemma \ref{traceablebimodule} that E is a traceable C-module.

Now, suppose two binary quadratic forms $(E, q, L)$ and $(E', q', L')$ are similar. By Proposition \ref{induceevenclifford} and Proposition \ref{semilinear}, it follows that the associated pairs $(C_0(E, q, L), E)$ and $(C_0(E', q', L'), E')$ are isomorphic as pairs. This gives the well-definedness of the map in the statement of Theorem \ref{MainTheorem}.
\subsection*{\textnormal{\textit{Surjectivity of the Map in Theorem \ref{MainTheorem}}}}
% \paragraph{\textbf{Step 2:}} \textbf{Surjectivity of the Map in Theorem \ref{MainTheorem}.}\\

Conversely, given a quadratic $\OS$-algebra $C$ and a \textit{traceable} $C$-module $E$ we can construct $\OS$-module $L=\wedge^2 E \tesnor (C/\OS)^\vee$ of rank 1 and a quadratic form $(E, q, L)$ such that there are isomorphisms $C_0(E, q, L)\cong C$ as $\OS$-algebras,  $C_1(E, q, L)\cong E$ as $\OS$-modules which are compatible with the $C$-module E and the $C_0(E, q, L)$-module $C_1(E, q, L)$.  We break this proof into two cases :
\subsubsection*{\textnormal{\textit{Case 1: If  $C$ and $E$ are free $\OS$-modules:}}}
    % \paragraph{\textbf{Case 1:}} \textbf{\underline{If  $C$ and $E$ are free $\OS$-modules :}\\
    
In this setting, Lemma \ref{freedirectsummand} allows us to select bases \( \{1, \tau\} \) for \( C \) and \( \{e_1, e_2\} \) for \( E \), respectively. Since \( \tau e_1 \) and \( \tau e_2 \) are elements of \( E \), which is a free \(R\)-module of rank 2, we may express them as  
 \begin{equation}\label{uvw}
\tau e_1=se_1 - te_2\quad\text{and}\quad\tau e_2=de_1+me_2.\end{equation} 
where \( s, t, d, m \in \Gamma(S, \OS)  \).
 This immediately yields the relation  
\[
(\tau - m)e_2 = d e_1.
\]

To facilitate computations, we perform a change of basis in \( C \), replacing \( \{1, \tau\} \) with \( \{1, \tau - m\} \). This transformation corresponds to the change-of-basis matrix  
\[
A =
\begin{pmatrix}
1 & -m \\
0 & 1
\end{pmatrix}.
\]
Thus, without loss of generality, we may assume that  
\begin{equation}\label{abc}
\tau e_1=be_1 - ae_2\quad\text{and}\quad\tau e_2=ce_1\end{equation}
where the new coefficients \( a, b, c \) are elements of \(\Gamma(S, \OS) \), obtained through the basis modification.

 If $\tau^2=m\tau +r$ with $m,r\in\OS$, then the matrices of $\OS$-linear endomorphism given by (left ) multiplication by $\tau$ on $C$ and $E$ are, respectively,
$$\left(\begin{matrix}
0 & r\\
1 & m \end{matrix}\right) ~~~and~~~
\left(\begin{matrix}
b & c\\
-a & 0 \end{matrix}\right).
$$
Then, the \textit{traceability} condition tells us that $m=b$. So we have $\tau^2=b\tau +r$. \\
Now multiplying \eqref{abc} by $\tau$ we get \\
$$\tau^2e_1=b\tau e_1-a\tau e_2$$ 
$$\Rightarrow (b \tau +r)e_1=b \tau e_1 - ace_1 $$\\
Now, comparing the coefficients of $e_1$, we get $r=-ac$. So, finally, we have 
\begin{equation}\label{mnp}
\tau^2=b\tau-ac \qquad \qquad\tau e_1=be_1-ae_2\qquad\qquad\tau e_2=ce_1\end{equation}

For this $(C, E)$, consider the binary quadratic form $q:E\rightarrow \OS$ defined by
\[q(xe_1 + ye_2) =ax^2 + bxy + cy^2,\] where $a,b,c$ are coming from \eqref{mnp} and $x, y\in \Gamma(S, \OS) $.

We claim that $C$ is isomorphic to the even Clifford algebra $C_0(E, q, \OS)$ as $\OS$-algebras. This claim is established through an appeal to the \textbf{universal property of the even Clifford algebra}, as articulated in Proposition \ref{universalC_0}.

For this consider the $\OS$-linear map $\phi: E \otimes E \rightarrow C$ defined by
 
   \[\begin{split} &\phi (e_i \otimes e_i)=q(e_i)\cdot 1_C  \text{~~~~for ~~~$i=1,2$};\\
&\phi (e_1 \otimes e_2)=\tau;\\
&\phi(e_2\otimes e_1)= b_q(e_1, e_2)\cdot 1_C - \tau. \end{split}\]
We need to check $\phi$ satisfies the following conditions:
\begin{equation}\label{stu}
\phi (x\otimes x)=q(x) \cdot 1_C  \hspace{3mm} and \hspace{3mm} \phi (x\otimes y) \cdot \phi (y\otimes z)=q(y) \phi (x \otimes z) \end{equation}
where $x,y,z\in E$. However, it is enough to verify on basis vectors.

Now \[\begin{split}
& \phi (e_1 \otimes e_2) \cdot  \phi (e_2 \otimes e_1)\\
& = \tau \cdot \{b_q (e_1, e_2)- \tau\}\\
& =\tau \cdot \{ q(e_1 + e_2)- q(e_1) - q(e_2) - \tau\}\\
&  = \tau \cdot \{a+ b + c -a - c - \tau \}\\
&  = b \tau - \tau^2\\
& = b \tau -b \tau + ac\\
&  = ac
\end{split}\]

Also, \[q(e_2) \phi (e_1 \otimes e_1)= q(e_2) q(e_1) = ac.\]

Then, by virtue of the universal property of the even Clifford algebra, there exists a unique $ \OS $-algebra morphism
\[
\psi : C_0(E, q, \OS) \to C.
\]
Since every basis vector in $\{1, \tau\}$ admits a preimage under the map $\psi$, it follows that $\psi$ is surjective. Furthermore, $\psi$ is an isomorphism of $\OS$-algebras, as both $C_0(E, q, O_S)$ and $C$ are free $\OS$-modules of the same rank. 
\subsubsection*{\textnormal{\textit{Case 2: If  $C$ or $E$ is not free as $\OS$-modules }}}
% \paragraph{\textbf{Case 2:}} \textbf{ \underline{If  $C$ or $E$ is not free as $\OS$-modules }:}\\  

In this case, for a given $(C, E)$ pair, consider the sheaf \begin{equation}\label{linebundle}
    L=\wedge^2 E\tesnor (C/\OS)^\vee.
    \end{equation}

 That this is a line bundle follows from  Lemma \ref{freedirectsummand}.

% \paragraph{Motivation}

% From Proposition \ref{quotientclifford}, we recall that if $q: E \to L$ is a binary quadratic form, then the quotient sheaf $C_0(E, q, L) / \OS$ is canonically isomorphic to $\wedge^2 E \otimes L^\vee$ as $\OS$-modules.

% This canonical isomorphism provides a natural identification:
% \[
% C / \OS \cong \wedge^2 E \otimes L^\vee.
% \]
% Dualizing both sides yields the relation
% \[
% \wedge^2 E \otimes (C / \OS)^\vee \cong L,
% \]
% which justifies the definition given in \eqref{linebundle}.

To construct a quadratic form, it follows from Lemma \ref{equivalent def} that it suffices to produce an $\OS$-linear morphism  \[\alpha: \mathrm{Sym}_2 E \rightarrow L,\] where $\mathrm{Sym}_2 E$ is the sheafification of the $\OS$-submodule generated by all sections of the form $(v\otimes v)$. Moreover, it is sufficient to define an $\OS$-linear morphism from $ (C/\OS \otimes \mathrm{Sym}_2E) $ to $ \wedge^2E$.

Consider the map $\alpha:C/\OS \otimes \mathrm{Sym}_2E \rightarrow  \wedge^2E$ defined on sections over $U \subseteq S$ by
 \[\alpha_U (\bar{\gamma} \otimes (e_1 e_2))= \gamma e_2 \wedge e_1.\] where $\bar{\gamma}\in \Gamma(U, C/\OS)$ and $e_1, e_2\in \Gamma(U, E).$
 
 This is a well-defined $\OS$-linear morphism so it corresponds to a binary quadratic form $q$; consider that binary form.

We aim to demonstrate that the even Clifford algebra $ C_0(E, q, L) $ associated with the quadratic form $ q $ is isomorphic to $ C $ as $ \mathscr{O}_S $-algebras. To establish this isomorphism, we appeal to the universal property satisfied by the even Clifford algebra $ C_0(E, q, L) $.

To proceed, we fix an open cover $ \{U_i\} $ of $ S $ such that, upon restriction to each $ U_i $, the vector bundle $ E|_{U_i} $ trivializes as $ \mathscr{O}_{U_i} x_i \oplus \mathscr{O}_{U_i} y_i $, and the line bundle $ L|_{U_i} $ trivializes as $ \mathscr{O}_{U_i} z_i $.

On each such open set $ U_i $, we then define a morphism  
\[
\theta_{U_i} \colon E|_{U_i} \otimes E|_{U_i} \otimes L^\vee|_{U_i} \to C|_{U_i}
\]  
by prescribing its action on basis elements, and extend it by $ \mathscr{O}_{U_i} $-linearity:
\[\begin{split} & \theta_{U_i}(x_i \otimes y_i \otimes z_i^\vee) = \left(0, (x_i \wedge y_i) \otimes z_i^\vee \right)\\
&\theta_{U_i} (x_i \otimes x_i \otimes z_i^\vee )= \left( z_i^\vee (q(x_i)),\ 0 \right) ;\\
&\theta_{U_i}(y_i \otimes x_i \otimes z_i^\vee) = \left( z_i^\vee(b_q(x_i, y_i)),\ - (x_i \wedge y_i) \otimes z_i^\vee \right).\end{split}\]

By remark \ref{localsplit} and \eqref{linebundle}, $\theta_{U_i}$ is well-defined. The local morphisms $ \theta_{U_i} $ glue canonically to a global morphism $ \theta \colon E \otimes E \otimes L^\vee \to C $; see \cite[Theorem 3.1]{sohamthesis} for the complete verification.

% Since $\OS$ is a direct summand of , it follows from \eqref{linebundle}  that  
% $$
% C \cong \OS \oplus \wedge^2 E \otimes L^\vee.
% $$
% We want to show that the even Clifford algebra $C_0(E, q, L)$ of that form $q$ is isomorphic to $C$ as  $\OS$-algebras. 
%  To verify that, we consider the $\OS$-linear map $\theta: E \otimes E \otimes L^\vee \rightarrow C$ defined on sections by:
%  \[\begin{split} &\theta (v \otimes v \otimes f)=f(q(v))\cdot 1_C\\
% &\theta (u \otimes v \otimes f)= (u \wedge v) \otimes f ;\\
% &\theta(v \otimes u \otimes f)= f(b_q(u, v))\cdot 1_C - \theta(u \otimes v \otimes f). \end{split}\]

Now, in order to invoke the universal property of the even Clifford algebra as stated in Proposition \ref{universalC_0}, it is necessary to verify whether the map $\theta$ satisfies the following conditions:
\[\begin{split}
& \theta(v\tensor v \tensor f) = f(q(v))\cdot 1_C ~~~~~~~\text{and}\\  
&\theta(u\tensor v \tensor f) \cdot \theta(v \tensor w \tensor g) = 
f(q(v)) \, \theta(u \tensor w \tensor g),
\end{split}\]
Since $\theta \in \Gamma(S, E^\vee \otimes E^\vee \otimes L \otimes C) $ , it is sufficient to check its local behaviour, which is already verified in Case 1.\\
So by the universal property of the even Clifford algebra, there exists a unique $\OS$-algebra morphism
\[\psi: C_0(E, q, L) \rightarrow C.\]
 We have seen that locally (Case 1), $\psi$ is an isomorphism; this implies $\psi$ is an isomorphism.
\subsection*{\textnormal{\textit{Injectivity of the map in Theorem \ref{MainTheorem}}}}
% \paragraph{\textbf{Step 3:}} \textbf{Injectivity of the map in Theorem \ref{MainTheorem}.}\\

It remains to show that the isomorphism classes of $(C, E)$ pairs correspond to similarity classes of binary forms. We will break this proof into two cases:
\subsubsection*{\textnormal{\textit{Case 1: If $C$ and $E$ are free $\OS$-modules:}}}
 % \paragraph{\textbf{Case 1:}}\textbf{\underline{If $C$ and $E$ are free $\OS$-module}:}\\

Let $E = e_1\OS \oplus e_2\OS$ and $C = \OS \oplus \tau\OS$, where $\tau = e_1 \cdot e_2$. The pair $(C, E)$ is associated with a quadratic form $q: E \to \OS$ defined as:
$$
q(xe_1 + ye_2) = ax^2 + bxy + cy^2,
$$
where $x, y, a, b, c \in \OS$, and the relation $\tau^2 = b\tau - ac$ holds.

Now, suppose $(C, E)$ and $(C', E')$ are isomorphic pairs of modules over $S$. This means there exist:
\begin{itemize}
    \item An $\OS$-algebra isomorphism $\varphi: C \to C'$,
    \item An $\OS$-linear isomorphism $\psi: E \to E'$,
\end{itemize}
such that for sections over any open subset $U \to S$, the compatibility condition holds:
$$
\psi_U(c \cdot x) = \varphi_U(c) \cdot \psi_U(x),
$$
where $c \in \Gamma(U, C)$ and $x \in \Gamma(U, E)$.

From this setup, we can write:
$$
E' = \psi_S(e_1)\OS \oplus \psi_S(e_2)\OS,
$$
and
$$
C' = \OS \oplus \psi_S(e_1) \cdot \psi_S(e_2)\OS,
$$
where $e_1, e_2 \in \Gamma(S, E)$. Our goal is to show that:
$$
\varphi_S(\tau) = d \cdot \psi_S(e_1) \cdot \psi_S(e_2),
$$
for some unit $d \in \OS^*$.

% \section*{2. Compatibility of $\varphi$ and $\psi$}

We have already established that $C \cong C_0(q)$ and $C' \cong C_0(q')$, where $C_0(q)$ denotes the even Clifford algebra associated with the quadratic form $q$. To proceed, let us assume:
$$
\varphi_S(e_1 \cdot e_2) = d \cdot \psi_S(e_1) \cdot \psi_S(e_2) + e,
$$
for some $d, e \in \OS$. By the compatibility of $\varphi$ and $\psi$, we must satisfy the following identity on sections:

\begin{equation}\label{god}
    \psi_S((e_1\cdot e_2) e_2)=\phi_S(e_1\cdot e_2) \psi_S(e_2).
    \end{equation}

% \subsection*{Step 1: Simplify $\psi_S((e_1 \cdot e_2)e_2)$}

Using the module structure of $E$ and the definition of $q$, we compute:
$$
\psi_S((e_1 \cdot e_2)e_2) = \psi_S(e_1e_2^2).
$$
Since $e_2^2 = q(e_2)$ by the quadratic form $q$, this becomes:
$$
\psi_S(e_1e_2^2) = \psi_S(e_1q(e_2)).
$$
By linearity of $\psi$, we have:
$$
\psi_S(e_1q(e_2)) = q(e_2)\psi_S(e_1).
$$
So \begin{equation}\label{shiv}
\psi_S((e_1 \cdot e_2)e_2) = q(e_2)\psi_S(e_1)
\end{equation}

% \subsection*{Step 2: Expand $\varphi_S(e_1 \cdot e_2)\psi_S(e_2)$}

Substituting $\varphi_S(e_1 \cdot e_2) = d \cdot \psi_S(e_1) \cdot \psi_S(e_2) + e$ into the right-hand side of \eqref{god} , we get:
$$
\varphi_S(e_1 \cdot e_2)\psi_S(e_2) = \big(d \cdot \psi_S(e_1) \cdot \psi_S(e_2) + e\big)\psi_S(e_2).
$$
Expanding this expression:
$$
\varphi_S(e_1 \cdot e_2)\psi_S(e_2) = d \cdot \psi_S(e_1) \cdot \psi_S(e_2)^2 + e \cdot \psi_S(e_2).
$$
Using the property $\psi_S(e_2)^2 = q'(\psi_S(e_2))$ in $E'$, this becomes:
\begin{equation}\label{bishnu}
\varphi_S(e_1 \cdot e_2)\psi_S(e_2) = d \cdot \psi_S(e_1) \cdot q'(\psi_S(e_2)) + e \cdot \psi_S(e_2).
\end{equation}

% \subsection*{Step 3: Equating (3.6), (3.7), and (3.8)}

From \eqref{god}, \eqref{shiv} and \eqref{bishnu} , we equate the two expressions for $\psi_S((e_1 \cdot e_2)e_2)$:
$$
q(e_2)\psi_S(e_1) = d \cdot q'(\psi_S(e_2))  \psi_S(e_1)  + e \cdot \psi_S(e_2).
$$
Since $\psi_S(e_1)$ and $\psi_S(e_2)$ are linearly independent in $E'$, the coefficients of $\psi_S(e_1)$ and $\psi_S(e_2)$ must separately vanish. This implies:
$$
e = 0.
$$

Thus, we conclude:
$$
\varphi_S(e_1 \cdot e_2) = d \cdot \psi_S(e_1) \cdot \psi_S(e_2).
$$
Now $d \in \OS^*$ follows from the fact that $\varphi_S$ is an isomorphism.

The computation shows that:
$$
\varphi(\tau) = d \cdot \psi(e_1) \cdot \psi(e_2), \text{ where } d \in \OS^*.
$$

For simplicity let us write $\psi(e_1)\cdot \psi(e_2)=\tau'$, then $\phi(\tau)=d\tau'.$

  Now \[ \begin{split}
      & \tau^2-b\tau+ac=0\\
      & \Rightarrow \phi(\tau^2)-b \phi(\tau)+ac=0\\
      & \Rightarrow\phi(\tau)^2-b\phi(\tau)+ac=0\\
      & \Rightarrow d^2\tau'^2-bd\tau'+ac=0\\
      & \Rightarrow \tau'^2-b/d \cdot \tau'+a/d \cdot c/d=0\\
  \end{split}\]

  Consider the binary form $q':E'\rightarrow \OS$ defind by
  \[q(x\psi(e_1)+y\psi(e_2))=a/d\cdot x^2+b/d \cdot xy+c/d\cdot y^2\]
  where $x,y,a,b,c,d\in \OS$ coming from above.

  Then, it is easy to see that the following diagram is commutative.

  \[\begin{tikzcd}[row sep = large, column sep = large]
 E \arrow[d, "\psi"'] \arrow[r, "q"] & \OS \arrow[d, "\cdot 1/d"] \\
 E' \arrow[ur, phantom, "\scalebox{1.5}{$\circlearrowleft$}" description]\arrow[r, "q'"] & \OS
 \end{tikzcd}\]

 This shows that $q$ and $q'$ are similar.
 \subsubsection*{\textnormal{\textit{Case 2: If  $C$ or $E$  are not free as $\OS$-modules:}}}

 % \paragraph{\textbf{Case 2:}} \textbf{\underline{If  $C$ or $E$  is not free as $\OS$-modules }:}\\ 
 
 Let $(C, E)$ and $(C', E')$ are isomorphic pairs i.e., there exist
 \[\begin{split} &\phi:C\rightarrow C', \text{~~~~~~~~~~~~~~~~~$\OS$-algebra isomorphism and}\\
 &\psi:E\rightarrow E', \text{~~~~~~~~~~~~~~~~$\OS$-linear isomorphism}
 \end{split}\]
 such that on  sections $U\subseteq S$, $\psi_U(c\cdot x)=\varphi_U(c)\cdot \psi_U(x)$, where $c\in \Gamma(U,C)$ and $x\in \Gamma(U, E)$.

 Now the pair $(C, E)$ associates to the binary form $q:E\rightarrow L$ corresponding to the $\OS$-linear map $\alpha:C/\OS \otimes \mathrm{Sym}_2E \rightarrow  \wedge^2E$ defined on sections over $U\subseteq S$ by
 \[\alpha_U (\bar{\gamma} \otimes (e_1 e_2))= \gamma e_2 \wedge e_1.\] where $\bar{\gamma}\in \Gamma(U, C/\OS)$, $e_1, e_2\in \Gamma(U, E)$ and $L=(C/\OS)^\vee \otimes \wedge^2E.$

 Similarly, the pair $(C', E')$ associates to the binary form $q':E'\rightarrow L'$ corresponding to the $\OS$-linear map $\beta:C'/\OS \otimes \mathrm{Sym}_2E' \rightarrow  \wedge^2E'$ defined on sections over $U\subseteq S$ by
 \[\beta_U (\bar{\gamma'} \otimes (e_1' e_2'))= \gamma' e_2' \wedge e_1'.\] where $\bar{\gamma}\in \Gamma(U, C'/\OS)$, $e_1' , e'_2\in \Gamma(U, E')$ and $L'=(C'/\OS)^\vee\otimes \wedge^2 E'.$

\[\begin{tikzcd}[row sep = large, column sep = large]
 C/\OS \otimes \mathrm{Sym}_2E \arrow[d, "\bar{\phi}\otimes \mathrm{Sym}_2(\psi)"'] \arrow[r, "\alpha"] & \wedge^2E \arrow[d, "\wedge^2\psi"] \\
 C'/\OS \otimes \mathrm{Sym}_2E' \arrow[ur, phantom, "\scalebox{1.5}{$\circlearrowleft$}" description]\arrow[r, "\beta"] & \wedge^2E'
 \end{tikzcd}\]

The aforementioned diagram exhibits commutativity due to the following sequence of equalities:
\[
\begin{aligned}
(\wedge^2 \psi \circ \alpha) (\overline{\gamma} \otimes e_1 e_2) &= \wedge^2 \psi(\gamma e_2 \wedge e_1) \\
&=  \psi(\gamma e_2) \wedge \psi(e_1) \\
&= \phi(\gamma) \psi(e_2) \wedge \psi(e_1).
\end{aligned}
\]
Furthermore, consider the composition:
\[
\begin{aligned}
\beta \circ (\overline{\varphi} \otimes \mathrm{Sym}_2(\psi))(\overline{\gamma} \otimes e_1 e_2) &= \beta(\overline{\varphi}(\overline{\gamma}) \otimes \psi(e_1) \psi(e_2)) \\
&= \phi(\gamma) \psi(e_2) \wedge \psi(e_1).
\end{aligned}
\]

Using the Hom–Tensor adjunction property, we obtain the following commutative diagram.

\[\begin{tikzcd}[row sep = large, column sep = large]
 \mathrm{Sym}_2E \arrow[d, "\mathrm{Sym}_2(\psi)"'] \arrow[r, "q"] & \mathscr{H}om(C/\OS, \wedge^2E)=L \arrow[d, "\wedge^2 \psi \circ  - \circ \bar{\phi}^{-1}"] \\
 \mathrm{Sym}_2E' \arrow[ur, phantom, "\scalebox{1.5}{$\circlearrowleft$}" description]\arrow[r, "q'"] & \mathscr{H}om(C'/\OS, \wedge^2E')=L'
 \end{tikzcd}\]

Since \(\varphi\) and \(\psi\) are isomorphisms, the composition \(\left(\wedge^2 \psi \circ - \circ \overline{\varphi}^{-1} \right)\) induces an $\OS$-module isomorphism between \(L\) and \(L'\). Consequently, the quadratic forms \(q\) and \(q'\) are similar.
 \subsection*{\textnormal{\textit{Discriminant--preserving property(up to sign):}}}
In Lemma \ref{discriminant-comparison}, we established the identity  
$$
\Delta(q) = -\Delta(C_0(E, q, L)).
$$  
This confirms that the bijection described above is discriminant-preserving up to sign.

Hence, the proof is complete.
\end{proof}
\vspace{.05in}
\begin{rmk}
    Here, we provide the rationale behind the conditions assumed in Theorem \ref{MainTheorem}. One might naturally be inclined to ask whether the similarity classes of binary quadratic forms can be completely characterized by their associated even Clifford algebras.  
 Indeed, in the case of ternary quadratic forms---those defined on vector bundles of rank 3---this is known to hold true, as demonstrated in \cite{tevbal} and \cite{voightcharacterizing}. However, we provide an example illustrating that the mere isomorphism of even Clifford algebras is insufficient to guarantee the similarity of two binary quadratic forms. See \cite[Example A.1, Appendix A]{sohamthesis} for a detailed construction.

    For any binary quadratic form $ q $, the degree zero part of the Clifford algebra, denoted $ C_0(q) $, constitutes a quadratic algebra, while the degree one part of the Clifford algebra, $ C_1(q) $, naturally assumes the structure of a traceable $ C_0(q) $-module (as established in Lemma \ref{traceablebimodule}). Furthermore, a key proposition from Max-Albert Knus’s book \textit{Quadratic and Hermitian Forms over Rings} (\cite[Proposition 7.1.1, Chapter IV, §7]{knus}) asserts that if two binary quadratic forms $ q $ and $ q' $ are similar, then there exists a unique isomorphism between $ C_0(q) $ and $ C_0(q') $ as $R$-algebras. Moreover, there is a unique $R$-module isomorphism between $ C_1(q) $ and $ C_1(q') $ that respects the respective $ C_0(q) $ and $ C_0(q') $-module structures.
\end{rmk}

\vspace{.05in}
\begin{rmk}
It is common in the literature to define the discriminant of a binary quadratic form $ q(x, y) = ax^2 + bxy + cy^2 $, locally, as $ \Delta(q) = b^2 - 4ac $; see, for instance, \cite{highcomp}. However, in this paper, we adopt the alternative convention $ \Delta(q) = 4ac-b^2$. As a consequence, the correspondence established in Theorem~\ref{MainTheorem} preserves the discriminant only up to sign.

Had we instead chosen the classical definition $ \Delta(q) = b^2 - 4ac $, the correspondence would yield an exact preservation of the discriminant. 
\end{rmk}

\vspace{.05in}
\begin{corollary}

 Let \( q_1: E_1 \to L_1 \) and \( q_2: E_2 \to L_2 \) be two binary quadratic forms over a scheme \( S \). If there exist:

\begin{enumerate}
    \item An isomorphism between \( C_0(E_1, q_1, L_1) \) and \( C_0(E_2, q_2, L_2) \) as \( \OS \)-algebras, and
    \item An isomorphism between \( E_1 \) and \( E_2 \) as \( \OS \)-modules, compatible with the \( C_0(E_1, q_1, L_1) \)-module (or \( C_0(E_2, q_2, L_2) \)-module) structure,
\end{enumerate}

then \( q_1 \) and \( q_2 \) are similar. Conversely, the similarity of \( q_1 \) and \( q_2 \) implies the existence of such isomorphisms.

\end{corollary}
\begin{proof}
    This follows  from Step 1 and Step 3 of the previous theorem.
\end{proof}

Now we want to classify isometry classes of binary forms.
\vspace{.05in}
\begin{corollary}

For any \( L \in \mathrm{Pic}(S) \), two binary quadratic forms \( q_1: E_1 \to L \) and \( q_2: E_2 \to L \) are isometric if and only if the pairs \( (C_0(E_1, q_1, L), E_1) \) and \( (C_0(E_2, q_2, L), E_2) \) are isomorphic pair, subject to the following conditions:
\[
(C_0(E_1, q_1, L)/\OS)^\vee \otimes \wedge^2 E_1 = L = (C_0(E_2, q_2, L)/\OS)^\vee \otimes \wedge^2 E_2,
\]
and
\[
\bar{\phi}^{-1}\otimes \wedge^2 \psi =Id,
\]
where \( \varphi \) and \( \psi \) are the isomorphisms related to the corresponding pairs.
\end{corollary}
\begin{proof}

    The "if" part follows from \cite[Chapter IV, Proposition 7.1.1, \S7]{knus}. The converse follows from the Theorem \ref{MainTheorem}.
\end{proof}
\vspace{.05in}
\begin{definition}
We introduce the term \emph{Clifford pair} for the pair $(C_0(E, q, L), C_1(E, q, L))$ associated with the binary quadratic form $(E, q, L)$.
\end{definition}

\begin{corollary}
     A Wood pair is a Clifford pair, and conversely, a Clifford pair is a Wood pair, over an arbitrary base scheme. 
 \end{corollary}
\begin{proof}
    In Theorem~\ref{MainTheorem}, we have shown that, over a base scheme $ S $, every Wood’s pair $ (C, E) $, where $ C $ is a quadratic $ \OS $-algebra and $ E $ is a traceable $ C $-module, arises as a Clifford pair $ (C_0(E, q, L), C_1(E, q, L)) $, for some binary quadratic form $ q \colon E \to L $ over $ S $.

Conversely, if $ (C_0(E, q, L), C_1(E, q, L)) $ is a Clifford pair, then by Lemma~\ref{traceablebimodule} it also satisfies the conditions defining a Wood’s pair.

Therefore, when $ S $ is an arbitrary scheme, the notions of Wood’s pair and Clifford pair coincide: each can be naturally identified with the other, and the correspondence is an equivalence of categories.
\end{proof}

\vspace{.05in}
\begin{rmk}
Given a binary quadratic form $ q \colon E \to L $ in the classical sense, there is, \emph{a priori}, no canonical method to construct a binary quadratic form on the dual module $ E^\vee $ within the same framework. However, by employing Wood's construction together with the main result of this paper, Theorem~\ref{MainTheorem}, we are able to establish such a construction in a natural and well-defined manner. We shall come back to this in Theorem \ref{duality}.
\end{rmk}

In order to study the behavior of primitive binary quadratic forms in a geometric setting, we now restrict the bijection of Theorem~\ref{MainTheorem} to such forms defined over an arbitrary base scheme $ S $.
\vspace{.05in}
\begin{definition}[\textbf{Primitive  Form}]\label{primidef}
A quadratic form $ q: E \to L $ is said to be locally primitive at a point $ p \in S $ if there exists an open neighborhood $ U $ of $ p $ such that the restrictions $ E|_U $ and $ L|_U $ are trivial as vector bundles, and the ideal 
\[
(q|_U)
\]
generated by the values of the restriction of $ q $ to $ U $ coincides with the entire structure sheaf $ \mathscr{O}_U $. In other words, we require that
\[
(q|_U) = \mathscr{O}_U.
\]

The form $ q $ is called primitive if this condition holds for every point $ p \in S $.
\end{definition}

\vspace{.05in}
\begin{proposition} \label{primitive}
A binary form $ q : E \to L $ is a primitive binary form if and only if $ E $ is a locally free $ C_0(E, q, L) $-module of rank 1.
\end{proposition}

\begin{proof}
Over an arbitrary commutative ring $ R $, the proof for primitive binary quadratic forms taking values in $ R $ itself is provided in \cite{KNESER}. As Definition~\ref{primidef} makes evident, primitivity is a local condition; consequently, the arguments presented in \cite{KNESER} extend naturally to the case of primitive forms taking values in line bundles over an arbitrary base scheme $ S $. A detailed treatment of this generalization can be found in \cite[Proposition~4.14]{sohamthesis}.
\end{proof}
\vspace{.05in}
 \begin{lemma}\label{locally free traceable}
Any locally free rank 1 $C$-module is a \textit{traceable} $C$-module.
\end{lemma}
\begin{proof}

As traceability is a property that can be verified locally, we may work under the assumption that the module is isomorphic to $ C $ over a suitable open subset. In such a case, the canonical algebra multiplication on $ C $ agrees with its structure as a module over itself. It follows that the corresponding trace maps are identical without further intervention.
\end{proof}
\vspace{.05in}
\begin{theorem}\label{classifi1}
We can restrict the bijection of Theorem~\ref{MainTheorem} to a bijection    
\[
\left\{
\parbox{2.5in}{\centering Similarity classes of primitive binary quadratic forms (E, q, L) over \( S \)}
\right\}
\longleftrightarrow
\left\{
\parbox{2.5in}{\centering Isomorphism classes of pairs \( (C, E) \), \\ with \( C \) a quadratic algebra over \( S \), \\ and \( E \) a locally free rank 1 \( C \)-module}
\right\}.
\]
\end{theorem}
\begin{proof}
    The result follows from a combination of Proposition~\ref{primitive}, Lemma~\ref{locally free traceable}, and Theorem~\ref{MainTheorem}.
\end{proof}

\section{Classification of Primitive Binary Quadratic Forms using Universal Norm Form}

While Theorem~\ref{classifi1} provides a classification of primitive binary quadratic forms over an arbitrary base scheme $ S $—based on the results of Theorem~\ref{MainTheorem}—we present here an alternative approach that generalizes the treatment of composition and structure developed by Kneser \cite{KNESER} and further refined by Bichsel and Knus \cite{bichselknus}. In \cite{bichselknus}, Bichsel and Knus provided a classification of non-degenerate, primitive binary quadratic forms over an arbitrary ring $R$, taking values in an invertible $R$-module by means of the universal norm form (see \cite[Example 4.2]{bichselknus}). In the present work, we extend and generalize their framework in several essential directions. Specifically, we consider: 
\begin{enumerate}
    \item A general base scheme $S$, rather than an affine base;
    \item Binary quadratic forms taking values in arbitrary line bundles over $S$; and
    \item Quadratic forms that are not necessarily non-degenerate—that is, we allow for possibly degenerate forms.
\end{enumerate}

One of the principal objectives of this work is to generalize Gauss composition to the setting of primitive binary quadratic forms defined over an arbitrary base scheme. By employing the following alternative construction, we not only succeed in extending the classical Gauss composition law to this broader geometric context, but we also gain greater conceptual clarity regarding the behaviour of the composition operation. Specifically, this approach makes it transparent that when two primitive binary quadratic forms $(E_1, q_1, L_1)$ and $(E_2, q_2, L_2)$, defined over a common base scheme, are composed, the resulting form naturally takes values in the tensor product of the respective line bundles, namely $L_1 \otimes L_2$. This structural insight---regarding the line bundle associated with the composed form---emerges naturally from our framework. In contrast, such a property is not immediately evident when using the method presented in Theorem~\ref{MainTheorem}, which, while valid, lacks the same level of geometric naturality. For a detailed exposition and proof of this result, we refer the reader to Proposition~\ref{compgauss}.

Let
\begin{itemize}
    \item \( S \) be a scheme,
    \item \( C \) be a quadratic  \( \mathscr{O}_S \)-algebra with a norm map
    \[
        n_C : C\longrightarrow \mathscr{O}_S,
    \]
    \item \( E \) be a left \( C \)-module (i.e., a quasi-coherent \( \mathscr{O}_S \)-module with a left \( C \)-action),
    \item \( J \) be a quasi-coherent \( \mathscr{O}_S \)-module.
\end{itemize}
\vspace{.05in}
\begin{definition}
    A morphism of \( \mathscr{O}_S \)-modules
\[
    q: E \longrightarrow J
\]
is called a norm form (with values in \( J \)) if it satisfies the following conditions:

\begin{enumerate}
    \item (Compatibility with scalar multiplication) For every open set \( U \subseteq S \), every \( x \in E(U) \), and \( a \in C(U) \), we have
    \[
        q_U(a \cdot x) = q_U(x) \cdot n_{C}(a).
    \]
    
    \item (Quadraticity) The \emph{polar form}
    \[
        b_q(x, y) := q(x + y) - q(x) - q(y)
    \]
    defines a bilinear map
    \[
        b_q : E \times E\longrightarrow J
    \]
    of \( \mathscr{O}_S \)-modules.
\end{enumerate}

\end{definition}
\medskip

A morphism of norm forms
\[
    (E, q, J) \longrightarrow (E', q', J')
\]
is a pair of morphisms
\[
    (\alpha, \varphi), \quad \alpha \in \operatorname{Hom}_{C}(E, E'), \quad \varphi \in \operatorname{Hom}_{\mathscr{O}_S}(J, J'),
\]
such that the following diagram commutes:
\[
\begin{tikzcd}
E \arrow{r}{q} \arrow{d}[swap]{\alpha} & J \arrow{d}{\varphi} \\
E' \arrow{r}{q'} & J'
\end{tikzcd}
\]
i.e.
\[
    q' \circ \alpha = \varphi \circ q.
\]

\begin{proposition}
     Let \( S \) be a scheme, let \( C \) be a quadratic  \( \mathscr{O}_S \)-algebra equipped with its norm morphism 
$n_{C} : C \longrightarrow \mathscr{O}_S$, and let $E$ be a quasi-coherent left $C$-module. Then there exists a triple \( (E, j_{E}, J(E)) \), where:
\begin{itemize}
    \item \( J(E) \) is a quasi-coherent \( \mathscr{O}_S \)-module,
    \item \( j_{E} : E \to J(E) \) is an \( \mathscr{O}_S \)-morphism,
\end{itemize}
such that for any norm form \( (E, q, J) \), there exists a unique morphism
\[
p : J(E) \longrightarrow J
\]
with
\[
q = p \circ j_{E}.
\]

\[
\begin{tikzcd}
E \arrow[r, "j_{E}"] \arrow[dr, "q"'] & J(E) \arrow[d, dashed, "p"] \\
& J
\end{tikzcd}
\]
\end{proposition}
\begin{proof}
A proof over an arbitrary commutative ring $R$ is given in \cite[Proposition 2.3]{bichselknus}, and further elaborated in \cite[Proposition 7.1.1, Chapter 3]{knus}. We now provide a sketch of the corresponding construction in the geometric setting, valid over a general base scheme $S$.

 Let $E$ be a quasi-coherent $\OS$-module. We begin by defining a presheaf \( \mathscr{R}_{E}^{\mathrm{pre}} \) on $S$ by setting, for each open subset $U\subseteq S$,
\[
    \mathscr{R}_{E}^{\mathrm{pre}}(U) := \bigoplus_{x \in E(U)} \mathscr{O}_S(U) \cdot e_x,
\]
the free \( \mathscr{O}_S(U) \)-module generated by the set \( E(U) \). Sheafifying this presheaf yields the free $\OS$-module sheaf
\[
    \mathscr{R}_{E} := \widetilde{\mathscr{R}_{E}^{\mathrm{pre}}}.
\]
Next, consider the direct sum sheaf
\[
    \mathscr{F} := \mathscr{R}_{E} \oplus (E \otimes_{\mathscr{O}_S} E),
\]
which is again quasi-coherent. 

We now define a presheaf \( \mathscr{Q}^{\mathrm{pre}} \) of \( \mathscr{O}_S \)-submodules of \( \mathscr{F} \) as follows: for each open set \( U \subseteq S \), let \( \mathscr{Q}^{\mathrm{pre}}(U) \subseteq \mathscr{F}(U) \) be the \( \mathscr{O}_S(U) \)-submodule generated by:
\begin{align*}
    & (e_{a x} - n_{C}(a) e_x,\; 0), \\
    & (e_{x+y} - e_x - e_y,\; -x \otimes y),
\end{align*}
for all \( x, y \in E(U) \), \( a \in C(U) \).

These relations ensure compatibility with the norm map and enforce bilinearity of the polar form.

Let $\mathscr{Q}\subseteq \mathscr{F}$ be the subsheaf of $\OS$-modules obtained by sheafifying $\mathscr{Q}^{\mathrm{pre}}$. This defines a well-defined \( \mathscr{O}_S \)-submodule sheaf of \( \mathscr{F} \). The expression \( \mathscr{F}(U)/\mathscr{Q}(U) \) defines a presheaf of \( \mathscr{O}_S \)-modules. However, this quotient may not be a sheaf. We define the universal norm sheaf to be the sheafification of this presheaf:
\[
    J(E) := \left( U \mapsto \mathscr{F}(U) / \mathscr{Q}(U) \right)^\sim.
\]

This ensures \( J(E) \) is a quasi-coherent \( \mathscr{O}_S \)-module.

Define the morphism
\[
    j_E : E \longrightarrow J(E)
\]
by sending a local section \( x \in E(U) \) to the class \( [e_x, 0] \in J(E)(U) \). It follows from a direct computation that $j_E$ is a norm form taking values in $J(E)$ and that it satisfies the requisite universal property. A complete account of this can be found in \cite[Proposition 3.6]{sohamthesis}.
    
\end{proof}

\begin{proposition}\label{propuninormform}
Let $(E, j_E, J(E))$ be the universal norm form corresponding to the left $C$-module $E$, then
\begin{itemize}
    \item[(a)] The image of the universal norm form $ j_E \colon E \to J(E) $ generates $ J(E) $ as an $ \mathscr{O}_S $-module; that is,
$$
J(E) = \langle j_E(E) \rangle_{\mathscr{O}_S}.
$$

\item[(b)] Let $ f: T \to S $ be a morphism of schemes. Then there exists a canonical isomorphism of $ \mathscr{O}_T $-modules:
$$
f^* J(E) \xrightarrow{\cong} J({f^* E})
$$
which is compatible with the universal norm maps in the sense that the following diagram commutes:
$$
\begin{tikzcd}
f^* E \arrow[d, "\mathrm{id}"'] \arrow[r, "j_E \otimes 1"] & f^* J(E) \arrow[d, "\sim"] \\
f^* E \arrow[r, "j_{f^* E}"'] & J({f^* E})
\end{tikzcd}
$$
or equivalently:
$$
j_{f^* E}(x \otimes 1) = j_E(x) \otimes 1.
$$
\end{itemize}
\end{proposition}
\begin{proof}
 Verifications are straightforward from the definition of the universal norm form; see \cite[Lemma 3.7 and Lemma 3.8]{sohamthesis} for the details.
\end{proof}

For a universal norm form \( j_E: E \to J(E) \), the target \(J(E)\) is not necessarily a line bundle in general. However, our interest lies in those cases where \(J(E)\) \emph{is} a line bundle. In particular, we shall demonstrate that if \(E\) is a locally free \(C\)-module of rank one, then \(J(E)\) is indeed a line bundle.

\begin{proposition}\label{uniline}
   
Let \( C \) be a quadratic  \( \mathscr{O}_S \)-algebra with norm morphism
\[
n_{C} : C \to \mathscr{O}_S,
\]
and let \( j_{C} : C \to J(C) \) be the universal norm form associated to the left \( C \)-module \( E = C \). Then:
\[
(C, j_{C}, J(C)) \cong (C, n_{C}, \mathscr{O}_S)
\]
as norm form triples.
 \end{proposition}
\begin{proof}
This verification is a straightforward consequence of the universal property; the details are given in \cite[Proposition 3.11]{sohamthesis}.
\end{proof}
\vspace{.05in}
\begin{rmk}
    The preceding Proposition~\ref{uniline} shows that if $E$ is a locally free  $C$-module of rank 1, then $J(E)$ is a line bundle. From now on, unless otherwise specified, $ E $ will always denote a locally free $ C $-module of rank 1.
\end{rmk}

Given that $ E $ is a locally free $C$-module of rank 1, the corresponding universal norm form takes values in a line bundle. In fact, since $ C $ is a quadratic $ \mathscr{O}_S $-algebra, this form is explicitly a line bundle-valued binary quadratic form. Consequently, the even Clifford algebra is well-defined, and our aim is to explicitly determine this even Clifford algebra associated to the universal norm form.

\vspace{.05in}
 \begin{lemma}\label{normCliffordalgebra}
Let $ C $ be a quadratic $ \OS $-algebra over a scheme $ S $, equipped with a norm map $ n_C: C \to \OS $. Then the even Clifford algebra associated to the norm form $ n $ is canonically isomorphic to $ C $ as $ \OS $-algebras:
$$
C_0(C, n_C, \OS) \xrightarrow{\cong} C.
$$
\end{lemma}
\begin{proof}
The verification is a straightforward consequence of the universal property of the even Clifford algebra (Proposition~\ref{universalC_0}); see \cite[Lemma 3.10]{sohamthesis} for the complete argument.
\end{proof}

\vspace{.05in}
\begin{proposition}
    Let \( C \) be a quadratic \( \mathscr{O}_S \)-algebra over a scheme \( S \), and let \( E \) be a locally free left \( C \)-module of rank 1. Let \( j_E: E \to J(E) \) be the universal norm form. Then:
\[
C_0(E, j_E, J(E) \cong C
\]
as \( \mathscr{O}_S \)-algebras.

\end{proposition}
\begin{proof}
The result follows from the universal property of the even Clifford algebra, formulated in Proposition~\ref{universalC_0}. A complete and detailed argument is provided in \cite[Proposition 3.14]{sohamthesis}.

\end{proof}

Our primary objective in this section is to classify \emph{primitive binary quadratic forms} over an arbitrary base scheme $ S $, taking values in \emph{line bundles}, up to \emph{similarity}, by means of \emph{Clifford invariants}.

\begin{lemma}\label{formsimuninorm}
    Let \( (E, q, L) \) be a primitive binary quadratic form over a scheme \( S \). Then \( (E, q, L) \) is similar to a universal norm form \( (E, j, J(E)) \), where \( E \) is a locally free rank 1 left module over a quadratic \( \mathscr{O}_S \)-algebra \( C \), \( j: E \to J(E) \) is the universal norm form, and \( J(E) \) is the universal norm module.
\end{lemma}
\begin{proof}
    Given the primitive binary quadratic form \( q: E \to L \), we have its generalized Clifford algebra:
\[
C(E, q, L) = C_0(E, q, L) \oplus C_1(E, q, L),
\]
where:
\begin{itemize}
  \item \( C_0(E, q, L) =: C \) is a quadratic \( \mathscr{O}_S \)-algebra,
  \item \( C_1(E, q, L)  \) is a locally free rank 1 right \( C \)-module.
\end{itemize}

We can identify \( C_1(E, q, L)\) and \(E \) by canonical \( \mathscr{O}_S \)-modules isomorphism; see Theorem~\ref{cliffordbimoduleunderlyingmodule} for details. 
Since \( q: E \to L  \) is a norm form, by the universal property of \( j \), there exists a unique morphism:
\[
\phi: J(E) \to L
\]
such that:
\[
q = \phi \circ j.
\]

\noindent
That is, the following diagram commutes:
\[
\begin{tikzcd}
E \arrow[r, "j"] \arrow[dr, "q"'] & J(E) \arrow[d, "\phi"] \\
& L
\end{tikzcd}
\]
Now \( L \) and \( J(E) \) are invertible \( \mathscr{O}_S \)-modules, and \( q \) is primitive, the image of \( j \) under \( \phi \) generates \( L \). Hence, \( \phi \) is a surjective morphism of line bundles, and therefore an isomorphism:
\[
\phi : J(E) \xrightarrow{\cong} L.
\]
As \( q = \phi \circ j \) and \( \phi \) is an isomorphism, we conclude that \( (E, q, L) \) is similar to \( (E, j, J(E)) \), i.e., they differ only by a change of target line bundle via an isomorphism. Thus,
\[
(E, q, L) \sim (E, j, J(E)).
\]
\end{proof}

\begin{lemma}\label{compareuninorm}
    Let \( S \) be a scheme. Suppose:
\begin{itemize}
    \item \( C, C' \) are  quadratic \( \mathscr{O}_S \)-algebras ,
    \item \( E \) is a quasi-coherent left \( C \)-module,
    \item \( E' \) is a quasi-coherent left \( C' \)-module.
\end{itemize}

Let:
\begin{itemize}
\item \( j : E \to J(E) \) be the universal norm form associated to \( E \) with respect to a fixed morphism \( n_C : C \to \mathscr{O}_S \),
    \item \( j' : E' \to J'(E') \) be the universal norm form for \( E' \) with respect to \( n'_{C'} : C' \to \mathscr{O}_S \).

\end{itemize}

Assume we are given:
\begin{enumerate}

\item A morphism of quasi-coherent \( \mathscr{O}_S \)-modules \( f : E \to E' \) which is an isomorphism.
    \item A morphism of quasi-coherent \( \mathscr{O}_S \)-algebras \( g : C \to C' \), also an isomorphism.
    \item Compatibility: For every open \( U \subseteq S \), and local sections \( x \in E(U), a \in C(U) \), we have:
    \[
    f(a \cdot x) = g(a) \cdot f(x).
    \]
   
\end{enumerate}

\medskip

 Under these assumptions, there exists a unique isomorphism of quasi-coherent \( \mathscr{O}_S \)-modules:
\[
F : J(E) \xrightarrow{\cong} J'(E')
\]
such that the following diagram commutes:
\[
\begin{tikzcd}
E \arrow{r}{f} \arrow{d}{j} & E' \arrow{d}{j'} \\
J(E) \arrow{r}{F} & J'(E')
\end{tikzcd}
\]
\end{lemma}
\begin{proof}
It is a direct consequence of the universal property of the universal norm form. A detailed and rigorous verification can be found in \cite[Lemma 3.16]{sohamthesis}. 
\end{proof}
\vspace{.05in}
\begin{proposition}\label{simuninorm}
Let $(E, q, L)$ and $(E', q', L')$ be two similar primitive binary quadratic forms valued in line bundles over an arbitrary base scheme $S$. Then their corresponding universal norm forms are also similar.
    \end{proposition}
    
\begin{proof}

Given that the forms $(E, q, L)$ and $(E', q', L')$ are similar, it follows that their associated universal norm forms meet all the hypotheses of Lemma~\ref{compareuninorm}. Consequently, by Lemma~\ref{compareuninorm}, these universal norm forms are similar as well.

\end{proof}

\begin{theorem}\label{simprimi}
    For any scheme $S$, the natural map
    \[(E, q, L)\mapsto (C_0(E, q, L), C_1(E, q, L))\]
    induces a discriminant-preserving (up to sign) \textemdash specifically, satisfying
$$
\Delta(q) = -\Delta(C_0(E, q, L)),
$$ 
bijective correspondence
\[
\left\{
\parbox{2.5in}{\centering Similarity classes of primitive binary quadratic forms (E, q, L) over \( S \)}
\right\}
\longleftrightarrow
\left\{
\parbox{2.5in}{\centering Isomorphism classes of pairs \( (C, E) \), \\ with \( C \) a quadratic algebra over \( S \), \\ and \( E \) a locally free rank 1 \( C \)-module}
\right\},
\]
which is functorial in $S$.

An isomorphism of pairs $(C , E)$ and $(C', E')$ is given by an isomorphism $C \cong C'$ of $\mathscr{O}_S$-algebras, and an isomorphism $E \cong E'$ of $\mathscr{O}_S$-modules that respects the $C$ and $C'$ module structures. Here for any binary quadratic form $ q:E \rightarrow L $, the objects $ C_0(E, q, L) $ and $ C_1(E, q, L) $ denote, respectively, the degree zero and degree one components of the generalized Clifford algebra associated to the binary quadratic form $ q $.

\end{theorem}

\begin{proof}
    We will proceed step by step.
    \subsection*{\textnormal{\textit{Construction and well-definedness of the Reverse Map}}}

    For a pair $(C, E)$, we consider the universal norm form associated to $(C, E)$, which is a binary quadratic form over an arbitrary base scheme $S$. This assignment is well-defined by Proposition~\ref{simuninorm}.
    \subsection*{\textnormal{\textit{Well-definedness and Injectivity of the map in Theorem~\ref{simprimi}}}}
    
% \noindent\textbf{Well-definedness and Injectivity of the map in Theorem~\ref{simprimi}.}\\

% \subsection*{Well-defindness and Injectivity of the map in Theorem~\ref{simprimi}}

Suppose that $(E, q, L)$ and $(E', q', L')$ are similar binary quadratic forms. Then, according to Proposition~\ref{semilinear}, the induced pairs $(C = C_0(E, q, L),\ C_1(E, q, L) = E)$ and $(C' = C_0(E', q', L'),\ C_1(E', q', L') = E')$ are isomorphic. This establishes the well-definedness of the map.

We now establish injectivity. Suppose that $(E, q, L)$ and $(E', q', L')$ are two primitive binary quadratic forms such that the induced pairs
$$
(C = C_0(E, q, L),\ C_1(E, q, L) = E) \quad \text{and} \quad (C' = C_0(E', q', L'),\ C_1(E', q', L') = E')
$$
are isomorphic. Then, by Lemmas~\ref{formsimuninorm} and~\ref{compareuninorm}, and using the transitivity of similarity, it follows that $(E, q, L)$ and $(E', q', L')$ are similar.
\subsection*{\textnormal{\textit{Surjectivity of the map in Theorem \ref{simprimi}}}}

% \noindent\textbf{Surjectivity of the map in Theorem \ref{simprimi}.}\\
% \subsection*{Surjectivity of the map in Theorem \ref{simprimi}}

Given a pair $(C, E)$, where $C$ is a quadratic algebra and $E$ is a locally free $C$-module of rank 1, the associated universal norm form $j: E \to J(E)$ defines a \emph{line bundle-valued primitive binary quadratic form}.
\subsection*{\textnormal{\textit{Functoriality and discriminant--preserving property (up to sign) of the map in Theorem \ref{simprimi}}}}

% \noindent\textbf{Functoriality and discriminant-preserving (up to sign) of the map in Theorem \ref{simprimi}.}\\

This follows from Proposition~\ref{propuninormform} and Lemma~\ref{discriminant-comparison}.

\end{proof}

One of the objectives of the present work is to extend the classical theory of Gauss composition to the setting of an arbitrary base scheme $S$. This generalization, which will be developed in later sections, relies on a careful analysis of the algebraic and geometric structures underlying binary quadratic forms. The proposition that follows provides an essential step in understanding the composition law in a more conceptual and geometrically meaningful way. 
 \vspace{.05in}
\begin{proposition}\label{compgauss}
    Let $ S $ be a scheme, let $ C $ be a quadratic $ \mathscr{O}_S $-algebra with norm map $ n_C \colon C \to \mathscr{O}_S $, and let $ E_1 $, $ E_2 $ be locally free $ C $-modules of rank 1. Denote by $ J(E_1) $ and $ J(E_2) $ the universal norm modules of $ E_1 $ and $ E_2 $, respectively, with associated universal norm forms  
$$
j_{E_1} \colon E_1 \to J(E_1), \quad j_{E_2} \colon E_2 \to J(E_2).
$$

Let $ E := E_1 \otimes_{C} E_2 $. Then the tensor product $ E $ is a locally free $ C $-module of rank 1, and there exists a canonical universal norm form  
$$
j_{E} \colon E \to J(E_1) \otimes_{\mathscr{O}_S} J(E_2),
$$  
defined by the composition  
$$
j_E(x \otimes y) = j_{E_1}(x) \otimes j_{E_2}(y),
$$  
for local sections $ x \in E_1 $, $ y \in E_2 $.
\end{proposition}
\begin{proof}
    We provide here a sketch of the proof, focusing on the key constructions and logical flow. Full details may be found in \cite[Proposition 3.19]{sohamthesis}.
    
Define the map
\[
j_E : E_1 \otimes_C E_2 \to J(E_1) \otimes_{\mathscr{O}_S} J(E_2), \quad m_1 \otimes m_2 \mapsto j_{E_1}(m_1) \otimes j_{E_2}(m_2).
\]
It is easy to see that it is a norm form. We now show that $j_E$ satisfies the universal property. Let $q : E \to J$ be an arbitrary norm form. We will construct a unique $\mathscr{O}_S$-linear map 
$$
p: J(E_1) \otimes_{\mathscr{O}_S} J(E_2) \to J
$$
such that $q = p \circ j_E$. Now
\begin{enumerate}
    \item For fixed $m_2 \in E_2(U)$, define $q_{m_2}: E_1 \to J$ by $q_{m_2}(m_1) = q(m_1 \otimes m_2)$. This is a norm form.
    \item By the universal property of $j_{E_1}$, there is a unique $p_{m_2}: J(E_1) \to J$ such that $q_{m_2} = p_{m_2} \circ j_{E_1}$.
    \item The assignment $m_2 \mapsto p_{m_2}$ defines a norm form $E_2 \to \mathscr{H}om_{\mathscr{O}_S}(J(E_1), J)$.
    \item By the universal property of $j_{E_2}$, there is a unique $\phi: J(E_2) \to \mathscr{H}om_{\mathscr{O}_S}(J(E_1), J)$ such that $\phi(j_{E_2}(m_2)) = p_{m_2}$.
    \item Define $p: J(E_1) \otimes_{\mathscr{O}_S} J(E_2) \to J$ by $p(t_1 \otimes t_2) = \phi(t_2)(t_1)$. Then
    \[
    p(j_E(m_1 \otimes m_2)) = \phi(j_{E_2}(m_2))(j_{E_1}(m_1)) = p_{m_2}(j_{E_1}(m_1)) = q(m_1 \otimes m_2).
    \]
\end{enumerate}
Uniqueness follows because $j_E(E)$ generates $J(E_1) \otimes_{\mathscr{O}_S} J(E_2)$. Thus, the universal norm module for $E$ is
\[
J(E) = J(E_1) \otimes_{\mathscr{O}_S} J(E_2).
\]
This follows since $j_E$ is a norm form with values in this sheaf, and the universal property ensures it is initial among all norm forms for $E$. The universal norm form for $E_1 \otimes_C E_2$ is
\[
j_E: E_1 \otimes_C E_2 \to J(E_1) \otimes_{\mathscr
{O}_S} J(E_2), \quad j_E(m_1 \otimes m_2) = j_{E_1}(m_1) \otimes j_{E_2}(m_2),
\]
with universal norm module $J(E_1) \otimes_{\mathscr{O}_S} J(E_2)$. 
\end{proof}

\section{Applications}
\subsection{Relationship to the Work of Wood}
In this section, we explain how our current work compares to  the important work of Wood \cite{Wood}. We start by reviewing the main ideas from \cite{Wood}. As we will see, the two main differences between the approach of Wood and our approach are that Wood uses \textit{linear} binary quadratic forms, whereas we use classical binary forms; furthermore, a conceptual distinction lies in our explicit engagement with the Clifford point of view, a perspective absent in Wood’s treatment.
\vspace{.05in}
\begin{definition}
    A \textit{linear} binary quadratic form over a scheme \( S \) is a triple \( (V, L, f) \), where \( V \) is a locally free \( \OS\)-module of rank 2, \( L \) is a locally free \( \OS\)-module of rank 1 (that is, a line bundle), and \( f \) is a global section of the \( \OS\)-module \( \operatorname{Sym}^2 V \otimes L \). 
    \end{definition}

Let \( U \subseteq S \) be an open subset over which the restrictions \( V|_U \) and \( L|_U \) are free.  
Let $\{x, y\}$ be a basis of \( V|_U \) over \( \mathscr{O}_U \), and let $\{z\}$ be a basis of \( L|_U \) over \( \mathscr{O}_U \).  
By a slight abuse of notation, omitting tensor products, the quadratic form \( q \) can be locally expressed on \( U \) as  

\[
q = a x^2 z + b x y z + c y^2 z,
\]

for some \( a, b, c \in \Gamma(U, \OS) \). This expression is sometimes denoted succinctly as \( [a, b, c] \), omitting explicit reference to the chosen bases.  
\vspace{.05in}
\begin{definition}[Isomorphism of \textit{linear} binary Quadratic Forms]  
  
Two \textit{linear} binary quadratic forms \( (V, L, q) \) and \( (V', L', q') \) are said to be \emph{isomorphic}  
if there exist isomorphisms of \( \OS \)-modules  
\[
f: V \xrightarrow{\simeq} V', \quad g: L \xrightarrow{\simeq} L'
\]
such that the induced map  
\[
\operatorname{Sym}^2 f \otimes g
\]
sends \( q \) to \( q' \). This equivalence can be understood as a \emph{right action} of the group  
\( \operatorname{GL}_2(V) \times \operatorname{GL}_1(L) \), which we refer to as the  
\emph{\( \operatorname{GL}_2 \times \operatorname{GL}_1 \)-action}.  
\end{definition}

Locally, this action can be described explicitly. Let \( U \subseteq S \) be an open subset  
over which both \( V|_U \) and \( L|_U \) are free. Since every section  
\( q \in \Gamma(U, \operatorname{Sym}^2 V \otimes_{\OS} L) \)  
can be expressed as a linear combination of simple tensors, it suffices to compute the action on elements of the form  
\( x_1 \otimes x_2 \otimes z \), where  
\[
x_1, x_2 \in \Gamma(U, V), \quad z \in \Gamma(U, L).
\]
Given a pair \( (\mu, \varepsilon) \in \Gamma(U, \operatorname{GL}_2(\OS) \times \operatorname{GL}_1(\OS)) \),  
the action is given by  
\[
(\mu, \varepsilon) \cdot (x_1 \otimes x_2 \otimes z) := (\mu x_1) \otimes (\mu x_2) \otimes (\varepsilon z).
\]
This describes how the \emph{\( \operatorname{GL}_2 \times \operatorname{GL}_1 \)-action} transforms \textit{linear} binary quadratic forms  
under changes of basis in \( V \) and \( L \).

Following Theorem 1.4 from Wood's work \cite{Wood}, there is a functorial discriminant preserving bijection:
\[
\left\{
\parbox{2.5in}{\centering Isomorphism classes of \textit{linear} binary quadratic forms over \( S \)}
\right\}
\longleftrightarrow
\left\{
\parbox{2.5in}{\centering Isomorphism classes of \( (C, E) \), \\ with \( C \) a quadratic algebra over \( S \), \\ and \( E \) a traceable \( C \)-module}
\right\}.
\]

\textit{Given } $(C, M)$ \textit{ and the corresponding } $f \in \mathrm{Sym}^2 V \otimes L$, \textit{we have } 
$M \cong V$ \textit{ as } $\OS$-\text{modules and}
\[
C / \OS \cong \Lambda^2 V^* \otimes L^*
\]
\textit{as } $\OS$-\text{modules. An isomorphism of pairs } $(C, M)$ \textit{ and } $(C', M')$ \textit{ is given by an isomorphism } $C \cong C'$ \textit{ of } $\OS$\text{-algebras, and an isomorphism } $M \cong M'$ \textit{ as } $\OS$-\text{modules that respects the } $C$ \textit{ (or } $C'$\text{) module structure.}

Here, we briefly recall the local version of the bijection described above; for a global formulation, we refer the reader to \cite{Wood}.

Given a \textit{linear} binary quadratic form \( f \in \mathrm{Sym}^2 V \otimes L \), we construct the modules \( C \) and \( M \) over \( \OS \) as follows:  
\[
C = \OS \oplus \wedge^2 V^\vee \otimes L^\vee, \quad M = V. 
\]  

To fully define the algebraic structure of \( C \) and the module structure of \( M \), we must now specify their respective operations. Initially, we consider the case where both \( V \) and \( L \) are free, such that \( V = \OS x \oplus \OS y \) and \( L = \OS z \). Redefining the basis, we let \( (1,0) \) and \( (0, (x^* \wedge y^*) \otimes z^*) \) in \( C \) correspond to \( 1 \) and \( \tau \), respectively. Given a quadratic form \( f = a x^2 z + b x y z + c y^2 z \), we designate \( 1 \) as the multiplicative identity of \( C \) and define the remaining algebra and module structures by the following relations:  
\[
\tau^2 = -b\tau - ac, ~~~~ \tau x = -cy - bx,~~~~~  \tau y = ax.
\]  

These structural relations endow \( M \) with the properties of a \textit{traceable} \( C \)-module.

Given a quadratic \(\OS\)-algebra \( C \) and a \textit{traceable} \( C \)-module \( M \), we can construct the \(\OS\)-modules \( V = M \) and  
\[
L = \Lambda^2 V^* \otimes (C / \OS)^*.
\]  
In the case where both \( C \) and \( M \) are free \(\OS\)-modules, we may select bases \( \{1, \tau\} \) for \( C \) and \( \{x, y\} \) for \( M \), satisfying the structural relations:  
\[
\tau x = -cy - bx, ~~~~ \tau y = ax,
\]
for some coefficients \( a, b, c \in \OS \). By adjusting \( \tau \) via an element of \(\OS\) if necessary, we can ensure that \( \tau y \) is a scalar multiple of \( a \), and we call such a basis \( \{1, \tau\} \) as normalized.  

If the quadratic relation  
\[
\tau^2 = -q\tau - r
\]  
holds for some \( q, r \in \OS \), then the \textit{traceability} condition imposes the constraint \( q = b \), while the requirement  
\[
\tau^2 = -q\tau - r
\]  
yields the additional condition \( r = ac \). Consequently, from the data of \( (C, M) \), we can construct the quadratic form  
\[
ax^2 z + bxyz + cy^2 z,
\]  
where \( z = x^* \wedge y^* \otimes \bar{\tau}^* \), and \( \bar{\tau} \) represents the image of \( \tau \) in \( C / \OS \).  

For an arbitrary pair \( (C, M) \), this construction determines an element \( f \in \mathrm{Sym}^2 V \otimes L \) locally on \( S \) in situations where \( C \) and \( M \) are free \(\OS\)-modules. To confirm that the local definitions of \( f \) are compatible across overlaps, it suffices to verify that any alternate choice of a basis for \( M \) and a corresponding normalized basis for \( C \) yields the same quadratic form.  

The constructions outlined above are mutually inverse, as their local definitions are explicitly constructed to be inverses. This establishes the bijection in Theorem 1.4 in \cite{Wood}.

There is a distinction between \textit{linear} binary quadratic forms and binary quadratic forms in the classical sense, such as those considered in the work of Kneser. Kneser considers binary quadratic forms \( q: M \to N \), where \( M \) is a locally free \( R \)-module of rank 2 and \( N \) is a locally free \( R \)-module of rank 1. These maps satisfy the properties that for all \( r \in R \) and \( m \in M \),  

\[
q(rm) = r^2 q(m),
\]  

and that the expression  

\[
q(x + y) - q(x) - q(y)
\]  

defines a bilinear form on \( M \times M \).

\begin{proposition}\label{comparison}
    Quadratic maps \( q: M \to N \) in the sense of Kneser, as described earlier, are in bijective correspondence with \textit{linear} binary quadratic forms \( f \in \mathrm{Sym}^2 M^* \otimes N \), where \( M \) and \( N \) are \( R \)-modules that are locally free of ranks \( 2 \) and \( 1 \) respectively. 
\end{proposition}

\begin{proof}
Refer to \cite[Proposition 6.1]{Wood} for the proof.
\end{proof}

\subsection*{Involutions and duality on binary forms}
\begin{theorem}\label{duality}
    $E$ be locally free rank 2, $L$ be locally free rank 1 over a general scheme $S$. Then we have
    \begin{itemize}
        \item  A duality between classical $L$-valued binary quadratic forms on $E$ and $E^\vee$;
        \item A duality between Wood's \textit{linear} binary $L$-valued quadratic forms on $E$ and $E^\vee$;
        \item Involutions on the space of Clifford pairs and on the space of Wood's pairs.
    \end{itemize}
    
\end{theorem}
\begin{proof}
Let $ S $ be a general scheme. Let $ E $ be a locally free sheaf of rank 2 over $ S $, and let $ L $ be a line bundle (i.e., a locally free sheaf of rank 1) over $ S $.

We consider two important locally free sheaves on $ S $ whose global sections are given by:
\[
\Gamma(S, \mathscr{Q}uad(E, L)): \text{\{classical binary quadratic forms on $E$ with values in $L$\}}
\]
\[
\Gamma(S, \mathscr{W}lbf(E, L)) =\Gamma(S, \mathrm{\Sym^2(E)} \otimes L): \text{\{Wood's \textit{linear} binary quadratic forms on $E$ with values in $L$\}}.
\]

We begin by investigating the local case i.e. when $S$=Spec($R$).
\begin{itemize}
    \item \textbf{Step 1}:\\
    Let $R$ be a ring, and $E=Re_1 \oplus Re_2$ a free rank 2 $R$-module. Consider the binary quadratic form $q:E\rightarrow R$ defined by: \[q(xe_1+ye_2)=ax^2+bxy+cy^2,\] where $x,y,a,b,c\in R$. The even Clifford algebra $C_0(M, q, R)$ is generated by $ <1, e_1\cdot e_2>$ as $R$-module. Let us write $\tau = e_1\cdot e_2 $. Then we have the following relations:
    $$\tau^2=b\tau-ac \qquad \qquad\tau e_1=be_1-ae_2\qquad\qquad\tau e_2=ce_1.$$
    \item \textbf{Step 2}:\\
    Write $C=C_0(E, q, R)$. For the pair $(C,E)$ with $$\tau^2=b\tau-ac \qquad \qquad\tau e_1=be_1-ae_2\qquad\qquad\tau e_2=ce_1.$$
    the associated Wood's \textit{linear} binary quadratic form $q_W:E\rightarrow R$ defined by:
    \[q_W(xe_1+ye_2)=cx^2-bxy+ay^2.\] where $x,y,a,b,c\in R$.
    \item \textbf{Step 3}:\\
    Now, we can view Wood's form as a quadratic form on $E^\vee$ in the sense of Kneser. Let us write this as $(q_W)_K:E^\vee\rightarrow R$, where  \[(q_W)_K(xe_1^\vee+ye_2^\vee)=cx^2-bxy+ay^2.\] with $x,y,a,b,c\in R$. Then  the following relations hold: $$\tau^2=-b\tau-ac \qquad \qquad\tau e_1^\vee=-be_1^\vee-ce_2^\vee \qquad\qquad\tau e_2^\vee=ae_1^\vee.$$
    \item \textbf{Step 4}:\\
    Now, treating the pair $(C, E^\vee)$ with the relations  $$\tau^2=-b\tau-ac \qquad \qquad\tau e_1^\vee=-be_1^\vee-ce_2^\vee \qquad\qquad\tau e_2^\vee=ae_1^\vee.$$ as a Wood's pair, we have the associated Wood's \textit{linear} binary quadratic form on $E^\vee$, $((q_W)_K)_W:E^\vee\rightarrow R$ which is defined as:
    \[((q_W)_K)_W(xe_1^\vee+ye_2^\vee)=ax^2+bxy+cy^2.\] where $x,y,a,b,c\in R$.

    \item \textbf{Step 5}:\\
    Again, this Wood's \textit{linear} binary quadratic form  corresponds to a binary quadratic form on $(E^\vee)^\vee \cong E$ in the sense of Kneser. Let us write this as $(((q_W)_K)_W)_K:E\rightarrow R$, defined by \[(((q_W)_K)_W)_K(xe_1+ye_2)= ax^2+bxy+cy^2 .\] where $x,y,a,b,c\in R$.
    \end{itemize}
So, we have $q= (((q_W)_K)_W)_K.$

Since all constructions are local in nature—defined in terms of modules over rings—and the sheaves $\mathscr{Q}uad(E, L)$ and $\mathrm{Sym}^2(E) \otimes L$ are coherent and compatible with restriction to open subsets, and since the bijective correspondence between Wood’s linear binary quadratic forms on a vector bundle and classical quadratic forms on the dual bundle is globally defined, these constructions naturally globalize to the scheme $S$.
\subsection*{\textnormal{\textit{Duality Between Forms on $ E $ and $ E^\vee $}}}
We have shown that:
\begin{itemize}
    \item Given a classical binary quadratic form $ q $ on $ E $, we can construct a corresponding one on $ E^\vee $ via the intermediate steps involving Wood’s form.
    \item This construction yields a natural set theoretical isomorphism between the sheaves of classical binary quadratic forms:
$$
\Gamma(S,\mathscr{Q}uad(E, L)) \xrightarrow{\sim} \Gamma(S, \mathscr{Q}uad(E^\vee, L))
$$
defined explicitly by the assignment:
$$
q \mapsto (q_W)_K
$$
\end{itemize}
    
Analogously, for Wood’s linear binary quadratic forms, we obtain a corresponding natural set theoretical isomorphism:
$$
\Gamma(S, \mathscr{W}lbf(E, L)) \xrightarrow{\sim} \Gamma(S, \mathscr{W}lbf(E^\vee, L))
$$
given by:
$$
Q \mapsto (Q_K)_W
$$
\subsection*{\textnormal{\textit{Involutions on Pairs}}}
 The foregoing calculations demonstrate the existence of an involution on Clifford pairs. By employing an analogous method, one can likewise construct an involution on Wood’s pairs.
 \end{proof}
 \vspace{.05in}
\begin{rmk}
    Given a binary quadratic form $q:E\rightarrow L$, there is no natural way to directly define a corresponding binary quadratic form $q'$ on $E^\vee$. However, by following the steps outlined above, we can construction a binary quadratic form $q'$ on $E^\vee$ that is associated with the binary form $q$ on $E$.
\end{rmk}

\subsection*{Dual Binary Forms and Wood's \textit{linear} binary Forms}
Projective duality is a foundational principle in projective geometry, embodying a profound symmetry between geometric entities of complementary dimensions—most notably, points and hyperplanes. In the plane, this manifests as a correspondence between points and lines, whereby incidence relations are preserved under an involution that interchanges their roles. This duality extends naturally to higher-degree objects, particularly conics, giving rise to the concept of the dual conic , which parametrizes the set of tangent lines to a given conic. As elucidated in \cite[Chapter 8]{dual}, the dual conic can be understood as the envelope of tangents to the original curve, and algebraically, for non-degenerate it corresponds to the inverse (up to a scalar) of the symmetric matrix defining the conic. In this work, we adapt and extend this classical framework to the setting of binary quadratic forms. We demonstrate that Wood’s formulation of linear binary quadratic forms—introduced in the context of orbit parametrizations and higher composition laws—arises precisely as the dual, in a well-defined sense, of classical binary quadratic forms.

Let \( K \) be an algebraically closed field with characteristic \( \neq 2 \).  
Let \( E = \langle e_1, e_2 \rangle \) be a free rank 2 \( K \)-module.  
A binary quadratic form \( q: E \to K \) is defined by:
\[
q(xe_1 + ye_2) = ax^2 + bxy + cy^2, \quad a, b, c, x, y \in K.
\]
This can be written in matrix form as:
\[
q(x, y) = \begin{pmatrix} x & y \end{pmatrix} 
\begin{pmatrix}
a & b/2 \\
b/2 & c
\end{pmatrix}
\begin{pmatrix} x \\ y \end{pmatrix},
\]
where \( C = \begin{pmatrix} a & b/2 \\ b/2 & c \end{pmatrix} \).  

The zero set is \( Z(q) = \{(x, y) \mid q(x, y) = 0\} \).

For \( p = (p_0, p_1) \in Z(q) \), the tangent line \( l_p \) is given by:
\[
\nabla q(p) \cdot \begin{pmatrix} x \\ y \end{pmatrix} = 0,
\]
where the gradient is:
\[
\nabla q(p) = (2ap_0 + bp_1, bp_0 + 2cp_1) = 2 C \cdot p.
\]
Thus, the tangent line equation is:
\[
l_p: (C \cdot p) \cdot \begin{pmatrix} x \\ y \end{pmatrix} = 0.
\]
\subsection*{\textnormal{\textit{ Dual Form for Nondegenerate \( q \)}}}

Assume \( q \) is nondegenerate (\( \det(C) \neq 0 \)).  
Let \( \xi = C \cdot p \). Substituting \( p = C^{-1}\xi \) into \( p^T C p = 0 \):
\[
p^T C p = 0 \implies (C^{-1}\xi)^T C (C^{-1}\xi) = 0 \implies \xi^T C^{-1} \xi = 0.
\]
The dual form is:
\[
q^*(u, v) = \begin{pmatrix} u & v \end{pmatrix} C^{-1} \begin{pmatrix} u \\ v \end{pmatrix}.
\]

\textit{ Dual Form}:  
The inverse matrix is:
\[
C^{-1} = \frac{1}{\det(C)} \begin{pmatrix} c & -b/2 \\ -b/2 & a \end{pmatrix}.
\]
Thus,
\[
q^*(u, v) = \frac{1}{\det(C)} \left( c u^2 - b uv + a v^2 \right).
\]
This is a similarity class of Wood’s \textit{linear} binary form associatated with  the classical binary quadratic form $q$.
\subsection*{\textnormal{ \textit{Degenerate Case: \( q(x, y) = (\alpha x + \beta y)^2 \)}}}

We are given a degenerate binary quadratic form:
$$
q(x, y) = (\alpha x + \beta y)^2
$$
and our goal is to define its dual, even though $ q $ is degenerate (i.e., has a double root), by approximating it with nearby non-degenerate forms and taking a limit of their duals.

This method follows the philosophy of algebraic geometry: understand singular objects as limits of smooth ones.
\subsection*{\textnormal{\textit{ The Degenerate Form}}}
Let:
$$
q(x, y) = (\alpha x + \beta y)^2 = \alpha^2 x^2 + 2\alpha\beta xy + \beta^2 y^2
$$

This is a degree-2 binary form with coefficients:
$$
[a : b : c] = [\alpha^2 : 2\alpha\beta : \beta^2]
$$

It is degenerate because its discriminant is zero:
$$
\Delta = b^2 - 4ac = (2\alpha\beta)^2 - 4(\alpha^2)(\beta^2) = 0
$$

So we cannot directly apply the standard duality operation. Instead, we will approximate this form with non-degenerate forms.
\subsection*{\textnormal{\textit{ Construct a One-Parameter Family}}}
Define a family of binary quadratic forms:
$$
q_t(x, y) = (\alpha x + \beta y)^2 + t \cdot g(x, y)
$$
where:
$$
g(x, y) = \gamma x^2 + \delta xy + \varepsilon y^2
$$
is a non-degenerate binary quadratic form, meaning:
$$
\delta^2 - 4\gamma\varepsilon \ne 0
$$

Now expand $ q_t $:
$$
q_t(x, y) = \alpha^2 x^2 + 2\alpha\beta xy + \beta^2 y^2 + t(\gamma x^2 + \delta xy + \varepsilon y^2)
$$

Group terms:
$$
q_t(x, y) = (\alpha^2 + t\gamma)x^2 + (2\alpha\beta + t\delta)xy + (\beta^2 + t\varepsilon)y^2
$$

So write:
$$
q_t(x, y) = a(t)x^2 + b(t)xy + c(t)y^2
$$
where:
\begin{align*}
a(t) &= \alpha^2 + t\gamma \\
b(t) &= 2\alpha\beta + t\delta \\
c(t) &= \beta^2 + t\varepsilon
\end{align*}

Note:
$$
\lim_{t \to 0} q_t(x, y) = (\alpha x + \beta y)^2 = q(x, y)
$$

Thus, $ q_t $ is a family of binary quadratic forms approaching $ q $ as $ t \to 0 $.
\subsection*{ \textnormal{\textit{Dual of Each Non-Degenerate Form}}}
For each $ t \ne 0 $, since $ q_t $ is non-degenerate, we can define its dual as:
$$
q_t^*(x, y) = c(t)x^2 - b(t)xy + a(t)y^2
$$

Substitute in the expressions:
$$
q_t^*(x, y) = (\beta^2 + t\varepsilon)x^2 - (2\alpha\beta + t\delta)xy + (\alpha^2 + t\gamma)y^2
$$

Now take the limit as $ t \to 0 $:
$$
q^*(x, y) = \lim_{t \to 0} q_t^*(x, y) = \beta^2 x^2 - 2\alpha\beta xy + \alpha^2 y^2
$$

Factor:
$$
q^*(x, y) = (\beta x - \alpha y)^2
$$
\subsection*{\textnormal{\textit{Final Result}}}
If:
$$
q(x, y) = (\alpha x + \beta y)^2
$$
then its dual is:
$$
q^*(x, y) = (\beta x - \alpha y)^2
$$

This is again a similarity class of Wood’s \textit{linear} binary form associatated with  the classical binary quadratic form $q$.

This result confirms that in both cases the dual form is a similarity class of Wood's linear binary form associatated to the classical binary quadratic form $q$.

Thus:
\begin{quote}
    ``Over algebraically closed fields of characteristic not 2, algebraic duality implies geometric duality, and vice versa.''
\end{quote}

\subsection{Quaternion Structure on Wood's Objects}
For a Wood's \cite{Wood} pair $(C, E)$, it is straightforward to see that $ W = C \oplus E $ carries the structure of a rank 4 $ \OS$-module. A natural question then arises: under what conditions does $ W $ admit the structure of a quaternion algebra?
\vspace{.05in}
\begin{proposition}
    For a Wood's pair $(C, E)$, where both $C$ and $E$ are free $\OS$-modules, $W$ carries the structure of a quaternion algebra.
\end{proposition}
\begin{proof}
In this case, using Theorem \ref{MainTheorem}, we have $ C \cong C_0(q) $ and $ E = C_1(q) $, where $ q \colon E \to \mathscr{O}_S $ is the binary quadratic form associated with the pair $ (C, E) $. Consequently, $ W \cong C_0(q) \oplus C_1(q) $. Now, by Example \ref{quaternion}, we know that $ C_0(q) \oplus C_1(q) $ carries the structure of a quaternion algebra; hence, $ W $ also inherits a quaternion algebra structure.

\end{proof}

\begin{proposition}
    For a Wood's pair $(C, E)$, if $C/\OS \cong \wedge^2E$, then $W$ has a quaternion algebra structure.
\end{proposition}
\begin{proof}

In this case, the binary quadratic form $ q $ associated with the pair $ (C, E) $ takes values in the trivial line bundle $ \mathscr{O}_S $. As discussed in Example \ref{quaternion}, for binary quadratic forms valued in the structure sheaf $ \mathscr{O}_S $, the graded algebra $ C_0(q) \oplus C_1(q) $ naturally carries the structure of a quaternion algebra over $ S $. In the proof of Theorem \ref{MainTheorem}, it was shown that $ C \cong C_0(q) $ and $ E = C_1(q) $. It follows that the object $ W $, constructed from this decomposition, inherits the structure of a quaternion algebra over the base scheme $ S $.
\end{proof}

\begin{proposition}
    Let  $(C, E)$ be a Wood's pair over an arbitrary base scheme $S$. Then, after a suitable faithfully flat base change, $W$ has a quaternion algebra structure.
\end{proposition}
\begin{proof}

In Theorem~\ref{MainTheorem}, we established the existence of a binary quadratic form \( q: E \to L \) defined over the base scheme \( S \), such that the associated even Clifford algebra $C_0(E, q, L)$ is isomorphic to the given algebra \( C \), and the Clifford bimodule $C_1(E, q, L)$ coincides with \( E \) under the canonical identification provided in Theorem~\ref{cliffordbimoduleunderlyingmodule}. However, in general, when quadratic forms take values in a non-trivial line bundle \( L \), the direct sum \( C_0(E, q, L) \oplus C_1(E, q, L) \) does not naturally inherit the structure of an  algebra. This subtlety arises due to the twisting by \( L \), which obstructs the usual multiplication rules from extending coherently across the graded components; for a detailed exposition, see~\cite{bichselknus}.

To circumvent this issue, we consider the Laurent algebra associated with \( L \), denoted
\[
\mathcal{L}(L) = \bigoplus_{n \in \mathbb{Z}} L^{\otimes n},
\]
which forms a graded \( \OS \)-algebra under the natural tensor product. By extending scalars to this Laurent algebra, we define
\[
W' = W \otimes_{\OS} \mathcal{L}(L),
\]
where \( W \) the direct sum \( C_0(E, q, L) \oplus C_1(E, q, L) \)  associated with \( (E, q, L) \). According to \cite[Lemma~3.1 and Lemma~3.2]{bichselknus}, this extended algebra \( W' \) becomes isomorphic to the classical Clifford algebra of a binary quadratic form \( q' \) taking values in a trivial line bundle. It is a well-established result in the theory of Clifford algebras that such classical Clifford algebras—associated with  quadratic forms over trivially twisted line bundles—naturally carry the structure of a quaternion algebra.
\end{proof}

\subsection{Recovering Picard Group -- Relation to the Work of Dallaporta }
Let $S$ be a scheme and $C$ be a quadratic $\OS$-algebra. A natural question arises: does there exist a locally free rank 2 $\OS$-module $E$ that is also a \textit{traceable} $C$-module?
\vspace{.05in}
\begin{proposition}\label{existence of traceable module}
Given any quadratic algebra $C$, there exists a \textit{traceable} $C$-module $E$.
    \end{proposition}
\begin{proof}
   From Lemma~\ref{locally free traceable}, it follows that every locally free $ C $-module of rank $ 1 $ is traceable. Consequently, any invertible $ C $-module --- that is, any object $ E \in \mathrm{Pic}(C) $ --- is a traceable $ C $-module.
\end{proof}

 Another natural question that arises is: given any quadratic algebra $C$, is there a binary form $q$ such that its generalized even Clifford algebra is isomorphic to $C$?
    
\vspace{.05in}
\begin{theorem}\label{existence of form}
    Given a quadratic $\OS$-algebra $C$ over an arbitrary base scheme $S$, there exists a binary quadratic form with values in a suitable line bundle, whose generalized even Clifford algebra is isomorphic to $C$.
        \end{theorem}
\begin{proof}
    Let $E \in$ Pic($C$). Then by Lemma \ref{locally free traceable}, $E$ is a \textit{traceable} $C$-module. Consider the pair $(C, E)$. By Theorem \ref{MainTheorem}, we obtain a binary form $q: E \rightarrow L$, for a suitable line bundle $L$ on $S$ such that $C_0(E, q, L)\cong C$. Alternatively, one may also proceed by following the approach outlined in Theorem~\ref{simprimi} using the universal norm form.
\end{proof}

In light of the proof of Theorem \ref{existence of form}, a natural  question emerges: Given a quadratic algebra $ C $ over the base scheme $ S $, how can one systematically parametrize its Picard group in terms of quadratic forms? This inquiry invites an exploration into the  interplay between the algebraic structure of $ C $ and the algebraic properties encoded by quadratic forms associated with $ S $. In his paper \cite[Theorem 3.24]{dallaporta}, William Dallaporta provided a parametrization of the Picard group associated with a given quadratic algebra $ C $. However, his approach relies on constructions developed in Wood's work \cite{Wood} (see also \cite[Definition 3.12]{dallaporta}), which does not incorporate the Clifford-theoretic framework. A fundamental distinction arises as Wood's formulation uses \textit{linear} binary quadratic forms instead of the classical notion of binary quadratic forms, as defined, for instance, by Kneser; see Proposition \cite[Proposition 6.1]{Wood} for further details. Consequently, Dallaporta's parametrization is rooted in \textit{linear} binary quadratic forms, which differ substantially from our point of view of using classical binary quadratic forms. Furthermore, both Dallaporta and Wood do not use the Clifford-theoretic perspective, which is central to ours. 
\vspace{.05in}
\begin{theorem}\label{Quotient Picard}
Let $ S $ be a general scheme and  $ C $ be a quadratic algebra over $ S $. The natural map 
\[(E, q, L)\rightarrow C_1(E, q, L)\]
induces
\[
\left\{
\begin{array}{c}
\text{Similarity classes of primitive binary quadratic forms (E, q, L) over } S \\
\text{having even Clifford algebra $C_0(E, q, L)$ isomorphic to } C
\end{array}
\right\}
\longleftrightarrow \mathrm{Pic}(C) / \sim,
\]
where for $ E, E' \in \mathrm{Pic}(C) $,  $ E \sim E' $ if and only if there exists an automorphism $ \varphi $ of $ C $ over $ S $ such that $ E' \cong \varphi^*E $ as $ C $-modules. Here, $ \varphi^*E $ denotes the $ C $-module $ E $ whose structure is locally given by $ \alpha \cdot e = \varphi(\alpha)e $, and  $\mathrm{Pic}(C) / \sim$ denotes the set of $\sim$-equivalence classes in $\mathrm{Pic}(C)$.
\end{theorem}

\begin{proof}
We will proceed step by step.
\subsection*{\textnormal{\textit{Step 1: Well-definedness of the map.}}}
 Suppose we start with a primitive binary quadratic form $ (E, q, L) $ over $ S $ having even Clifford algebra isomorphic to $ C $. By Proposition \ref{primitive}, the module $ C_1(E, q, L) = E $ is a locally free rank 1 $ C $-module. Hence, $ E \in \mathrm{Pic}(C) $.

Now, consider two similar primitive binary quadratic forms $ (E_1, q_1, L_1) $ and $ (E_2, q_2, L_2) $ over $ S $, both having isomorphic even Clifford algebra $ C $. From Theorem \ref{MainTheorem} , the pairs $ (C, E_1) $ and $ (C, E_2) $ are isomorphic as pairs. By the definition of isomorphic pairs, it follows that $ E_1 \sim E_2 $ in $ \mathrm{Pic}(C) $.
\subsection*{\textnormal{\textit{Step 2: Surjectivity of the map.}}}
 Conversely, suppose $ E \in \mathrm{Pic}(C) $. Then $ E $ is a locally free rank 1 $ C $-module. By Lemma \ref{locally free traceable}, $ E $ is a \textit{traceable} $ C $-module. For the pair $ (C, E) $, we can associate a binary quadratic form using Theorem \ref{MainTheorem}. Alternatively, one may also proceed by following the approach outlined in Theorem~\ref{simprimi}. Furthermore, the corresponding binary quadratic form is primitive, as shown in Proposition \ref{primitive}.
\subsection*{\textnormal{\textit{Step 3: Injectivity of the map.}}}
Now, if $ E_1 \sim E_2 $ in $ \mathrm{Pic}(C) $, then by the definition of the equivalence relation $ \sim $, the pairs $ (C, E_1) $ and $ (C, E_2) $ are isomorphic. By Theorem \ref{MainTheorem} , the corresponding primitive binary quadratic forms are similar.

Thus, we have established the desired bijection.
\end{proof}

Given a quadratic algebra \( C \), Theorem \ref{Quotient Picard} nearly parametrizes its Picard group, with the main obstacle being the presence of non-trivial automorphisms of \( C \).  To address this issue, we introduce a rigidification of quadratic algebras by eliminating these automorphisms via the notion of orientation. When 2  is not a zero divisor in $\Gamma(S, \OS)$, this approach is sufficient to fully recover the Picard group described in Theorem~\ref{Quotient Picard}.

\begin{definition}[\textbf{Oriented Quadratic Algebra}]
\label{def:oriented-quadratic-algebra}
Following \cite[Definition 3.2]{dallaporta}, let $ N $ denote a locally free $\OS$-module of rank 1 over a base scheme $ S $. An \emph{$ N $-oriented quadratic algebra} is defined as a pair $(C, \theta)$, where:
\begin{enumerate}
    \item $C$ is a quadratic $\OS$-algebra, meaning that $C$ is an $\OS$-algebra equipped with a structure such that the natural map $\OS \to C$ makes $C$ a finitely generated, locally free $\OS$-module of rank 2.
    
    \item $\theta: C/\OS \xrightarrow{\simeq} N^\vee$ is an isomorphism of $\OS$-modules, referred to as the \emph{orientation} of $C$. Here, $N^\vee=\mathscr{H}om_{\OS}(N, \OS)$ denotes the dual module of $N$, and $C/\OS$ represents the quotient module obtained by modding out the image of the structural morphism $\OS \to C$.
\end{enumerate}
\end{definition}
The orientation $\theta$ provides additional structure on $C$ by identifying the quotient $C/\OS$ with the dual of the rank-1 module $N$, thus encoding compatibility between the algebraic structure of $C$ and of $N$.

Given two $N$-oriented quadratic algebras $(C, \theta)$ and $(C', \theta')$, an \emph{isomorphism of oriented quadratic algebras} is an isomorphism $\psi: C \xrightarrow{\simeq} C'$ of $\OS$-algebras such that $\theta = \theta' \circ \overline{\psi}$, where $\overline{\psi}: C/\OS \xrightarrow{\simeq}  C'/\OS$ is the induced map on the quotients derived from $\psi$. This compatibility condition ensures that the orientations $\theta$ and $\theta'$ are preserved under the isomorphism $\psi$.

\vspace{.05in}
\begin{proposition}\cite[Proposition~3.6]{dallaporta}\label{Automorphism}
Let $S$ be a scheme such that $2$ is not a zero divisor in $\Gamma(S, \OS)$, and let $N \in \mathrm{Pic}(S)$. Then the group of automorphisms of an $N$-oriented quadratic algebra over $S$ is trivial.
\end{proposition}

\vspace{.05in}
\begin{definition}[\textbf{Twisted Quadratic Form}]
Let \( N \in \operatorname{Pic}(S) \). An \( N \)-twisted binary quadratic form, represented as a pair \( (E, q) \), is the binary quadratic form obtained by setting \( L = \Lambda^2 E \otimes_ \OS N \) in Definition \ref{quad form}. Specifically, \( q \) is a global section of  $\mathscr{Q}uad(E, \Lambda^2 E \otimes_ \OS N)$.
% \( \operatorname{Sym}^2 E^\vee \otimes_\OS \Lambda^2 E \otimes_ \OS N \).
    
\end{definition}
\vspace{.01in}
\begin{definition}[\textbf{Twisted Primitive Binary Form}]
Let \( N \in \operatorname{Pic}(S) \). An \( N \)-twisted binary quadratic form is said to be primitive if it is primitive when viewed as a binary quadratic form.
\end{definition}
\vspace{.05in}
\begin{theorem}
Let \(N \in \mathrm{Pic}(S)\). There is a functorial bijection
\[
\left\{
\begin{array}{c}
\text{Similarity classes of} \\
\text{primitive } N\text{-twisted binary} \\
\text{quadratic forms over } S
\end{array}
\right\}
\longleftrightarrow
\left\{
\begin{array}{c}
\text{Isomorphism classes of } (C, \theta, E), \\
\text{with } (C, \theta) \text{ an } N\text{-oriented} \\
\text{quadratic algebra over } S, \\
\text{and } E \text{ an invertible } C\text{-module}
\end{array}
\right\}.
\]
An isomorphism \((C, \theta, E) \cong (C', \theta', E')\) is a pair \((\psi, \phi)\) where \(\psi: (C, \theta) \cong (C', \theta')\) is an isomorphism of \(N\)-oriented algebras, and \(\phi: E \otimes_C C' \cong E'\) is an isomorphism of \(C'\)-modules.
\end{theorem}
\begin{proof}
This result is a direct consequence of Proposition \ref{quotientclifford} and Theorem \ref{MainTheorem}.
\end{proof}
\vspace{.05in}
\begin{theorem}\label{Picard}
Let $S$ be a general scheme for which $2$ is not a zero divisor in $\Gamma(S, \OS)$. Let $C$ be a quadratic algebra over $S$. There exists a set-theoretic bijection between:
\[
\left\{
\begin{array}{c}
\text{Similarity classes of primitive } N\text{-twisted binary quadratic forms over } S, \\
\text{whose even Clifford algebra is isomorphic to } C, \\
\text{and which preserve a fixed } N\text{-orientation}
\end{array}
\right\}
\longleftrightarrow{ \operatorname{Pic}(C)}.
\]
\end{theorem}
\begin{proof}

Let \( N \) be a chosen representative of the isomorphism class of \( (C / \OS)^\vee \). The similarity classes of binary forms whose even Clifford algebra is isomorphic to \( C \) and that preserve the fixed \( N \)-orientation induce an \( N \)-oriented automorphism of \( C \). By Proposition \ref{Automorphism}, the only automorphism of an \( N \)-oriented quadratic algebra over \( S \) is the identity.  

Consequently, if \( (E, q, L) \) and \( (E', q', L') \) are two similar \( N \)-twisted binary forms, then \( E \sim E' \) in the sense of Theorem \ref{Quotient Picard} if and only if \( E \) and \( E' \) are isomorphic as \( C \)-modules. The remainder of the proof follows directly from the reasoning established in the proof of Theorem \ref{Quotient Picard}.
\end{proof}

\subsubsection{Relationship to the Work of Kneser}
Let $ R $ be a commutative ring with unity, and let $ C $ be a quadratic $ R $-algebra. A \emph{binary quadratic form} over $ R $ is defined as a projective $ R $-module $ M $ of rank 2 together with a quadratic form $ q: M \to R $. The composition theory of such forms is intimately connected to the structure of the \emph{even Clifford algebra} $ C_0(M) $, which itself naturally acquires the structure of a quadratic $ R $-algebra.

In his seminal work~\cite{KNESER}, Kneser introduces and investigates two fundamental groups associated with $ C $. All definitions presented here are drawn directly from~\cite{KNESER}.
\vspace{.05in}
\begin{definition}[\textbf{Binary Quadratic Module of Type  C }]\label{typeC}
A binary quadratic module $ (M, q) $ is said to be of type $ C $ if:
\begin{itemize}
    \item $ M $ is a projective $ C $-module of rank 1,
    
\item The quadratic form satisfies $ q(cx) = n_C(c)\cdot q(x) $ for all $ c \in C, x \in M $, where $ n_C: C \to R $ is the norm map.
\end{itemize}
\end{definition}
\vspace{.05in}
\begin{definition}[\textbf{Isomorphism in  G(C) }]
An isomorphism between two elements $(M, q)$ and $(M', q')$ in  $G(C)$ is a $C$-linear isomorphism:
$$
\varphi: M \to M'
$$
such that the following diagram commutes:

$$
\begin{tikzcd}
M \arrow[rd, "q"'] \arrow[rr, "\varphi"] & & M' \arrow[ld, "q'"] \\
& R &
\end{tikzcd}
$$

This means that for all $x \in M$, we have:
$$
q'(\varphi(x)) = q(x)
$$

In other words, $\varphi$ preserves the quadratic structure — it is a $C$-linear isometry.

\end{definition}
\vspace{.05in}
\begin{definition}[\textbf{Group  G(C) }]
The group $ G(C) $ is defined to be the set of \emph{isomorphism classes} of \emph{primitive binary quadratic modules} $ (M, q) $ of type $ C $, where $ C $ is a fixed quadratic $ R $-algebra. The group operation is induced by the \emph{tensor product over $ C $}, defined as follows: for two such modules $ (M_1, q_1) $ and $ (M_2, q_2) $, their product in $ G(C) $ is given by

$$
(M_1, q_1) \cdot (M_2, q_2) := (M_1 \otimes_C M_2,\, q_1 \otimes q_2),
$$
where $ q_1 \otimes q_2 $ denotes the quadratic form on the tensor product module $ M_1 \otimes_C M_2 $, naturally induced from $ q_1 $ and $ q_2 $.

The identity element of this group is the isomorphism class of the pair $ (C, n_C) $, where $ n_C $ is the norm form on $ C $, defined by $ n_C(c) = c \cdot \overline{c} $, with $ \overline{c} $ being the image of $ c $ under the involution of $ C $.

Inverses in $ G(C) $ are given by \emph{involution}. Specifically, the inverse of an element represented by $ (M, q) $ is the class of the conjugate module $ (\overline{M}, \overline{q}) $, where:

    \begin{itemize}
        \item $ \overline{M} $ is the same underlying $ R $-module as $ M $, but equipped with a twisted $ C $-module structure via the involution $ \sigma: C \to C $, defined by $ \sigma(c) = \overline{c} $. That is, the action of $ C $ on $ \overline{M} $ is given by $ c \cdot x = \sigma(c)x = \overline{c}x $ for all $ c \in C $, $ x \in M $.

\item  The quadratic form $ \overline{q} $ on $ \overline{M} $ is defined by $ \overline{q}(x) = q(x) $; that is, it is the same function as $ q $, but now considered on the involuted module $ \overline{M} $.
\end{itemize}

Although the $ C $-module structure has changed in passing to $ \overline{M} $, the values of $ \overline{q} $ remain in $ R $, and the quadratic form still satisfies the compatibility condition:

$$
\overline{q}(c \cdot x) = n_C(c) \cdot \overline{q}(x),
$$
which holds because:
$$
\overline{q}(c \cdot x) = q(\overline{c} \cdot x) = n_C(\overline{c}) \cdot q(x) = n_C(c) \cdot q(x) = n_C(c) \cdot \overline{q}(x),
$$
using the fact that the norm is invariant under conjugation: $ n_C(\overline{c}) = n_C(c) $.

Recall that since $ C $ is a quadratic $ R $-algebra, it comes equipped with the unique \emph{involution}:

$$
\sigma: C \to C, \quad \sigma(c) = \overline{c},
$$
satisfying the identities:
$$
c + \overline{c} = \tr_C(c), \quad c \cdot \overline{c} = n_C(c),
$$
for all $ c \in C $, where $ \tr_C $ and $ n_C $ denote the trace and norm maps associated with the algebra $ C $, respectively.
\end{definition}
Kneser shows in \cite[Section 6, Theorem 3]{KNESER} that this structure ensures that the set of isomorphism classes of primitive binary quadratic modules over $ C $ forms a well-defined abelian group under the tensor product operation, denoted $ G(C) $. Further, Kneser provides  a natural exact sequence:
$$
1 \to R^\times / N(C^\times) \xrightarrow{\alpha} G(C) \xrightarrow{\beta} \mathrm{Pic}(C) \xrightarrow{\gamma} \mathrm{Pic}(R).
$$

    However, the homomorphism from $ G(C) $ to $ \mathrm{Pic}(C) $ is, in general, neither injective nor surjective. To address this, a new group, denoted $ H(C) $, was introduced by Kneser in \cite[Section 6, Proposition 2]{KNESER}, in order to obtain an isomorphism with $ \mathrm{Pic}(C) $. 

\vspace{.05in}
\begin{definition}[Triple in $ H(C) $]
An element of $ H(C) $ is an isomorphism class of triples $ (M, q, N) $, where:

    \begin{itemize}
        \item $ M $ is an invertible $ C $-module,
    
\item  $ N $ is an invertible $ R $-module,
    
\item  $ q: M \to N $ is a quadratic map satisfying:
        $$
        q(ax) = a^2 q(x), \quad b(x, y) = q(x + y) - q(x) - q(y) \text{ is } R\text{-bilinear},
        $$
        and
        $$
        q(cx) = n_C(c)\cdot q(x), \quad \forall c \in C, x \in M.
        $$
 \end{itemize}
\end{definition}
\vspace{.05in}
\begin{definition}[Isomorphism in $ H(C) $]
Two triples $ (M, q, N) $ and $ (M', q', N') $ are isomorphic in $ H(C) $ if there exist $ C $-linear isomorphism $ \varphi: M \to M' $ and $ R $-linear isomorphism $ \psi: N \to N' $ such that the following diagram commutes:
\[
\begin{tikzcd}
M \arrow[r, "q"] \arrow[d, "\varphi"'] & N \arrow[d, "\psi"] \\
M' \arrow[r, "q'"] & N'
\end{tikzcd}
\]
\end{definition}
\vspace{.05in}
\begin{definition}[Group Operation in $ H(C) $]
For $ (M_1, q_1, N_1), (M_2, q_2, N_2) \in H(C) $, define:
$$
(M_1, q_1, N_1) \cdot (M_2, q_2, N_2) = (M_1 \otimes_C M_2,\ q_1 \otimes q_2,\ N_1 \otimes_R N_2).
$$
\end{definition}

In order to establish the relationship between our work and Kneser's, we begin by clarifying how binary quadratic forms relate to binary quadratic modules of type $ C $.

\vspace{.05in}
\begin{proposition}
Let $ q : M \to N $ be a primitive binary quadratic form over a ring $ R $. Then $ q $ is of type $ C_0(M, q, N) $.
\end{proposition}

\begin{proof}
By Proposition \ref{primitive},  $ M $ is a projective module of rank 1 over the even Clifford algebra $ C := C_0(M, q, N) $.

Now consider the action of $ C $ on $ M $ via Clifford multiplication. For any $ c \in C $ and $ x \in M $, we compute:
$$
q(cx) = (cx)^2 = cx \cdot cx = c x c x.
$$

Using the canonical involution $ \sigma_C $ on the even Clifford algebra, this becomes:
$$
c \sigma_{C}(c) \cdot x^2 = c \sigma_{C}(c) \cdot q(x) =n_{C}(c) \cdot q(x),
$$
where $ n_{C}(c) := c \sigma_{C}(c) $ denotes the norm map associated with $ C $.

Hence, the quadratic form $ q $ satisfies the compatibility condition required by Definition \ref{typeC}, showing that the binary quadratic module $ (M, q, N) $ is indeed of type $ C_0(M, q, N) $. This completes the proof.
\end{proof}

In his seminal work \cite[Theorem 3, Section 6]{KNESER}, Kneser proved that the set $ G(C) $, associated with a given type $ C $, forms a group under composition. However, in that setting, he specifically considered the composition of \emph{primitive binary quadratic forms} of the \textit{same} type; that is, he examined the composition of two forms of type $ C_1 $ and $ C_2 $ where $ C_1 = C_2 = C $. 

In \cite[Proposition 2, Section 6]{KNESER}, he briefly remarks that the set $ H(C) $ also possesses a group structure by invoking the same line of reasoning used in \cite[Theorem 3, Section 6]{KNESER}. In order to draw a connection with Kneser’s work \cite{KNESER}, we consider here a more general notion of composition: namely, the composition of two types $ C_1 $ and $ C_2 $, under the condition that both $ C_1 $ and $ C_2 $ are isomorphic to a fixed type $ C $. It is important to note that this isomorphism need not be the identity map, which was implicitly assumed in Kneser’s original treatment.

Our approach generalizes Kneser's construction in two significant directions:

   \begin{itemize} 
\item  by explicitly incorporating isomorphisms between quadratic algebras of the \textit{same} type, and
    
\item  by allowing for arbitrary base schemes.
\end{itemize}
\vspace{.05in}
\begin{definition}
    We define the set appearing on the left-hand side of Theorem \ref{Picard} as $H(C)$. 
\end{definition}
% Although we retain the notation $ H(C) $ for the resulting structure, its meaning is now broader. It encompasses compositions involving non-identity isomorphisms between types and is defined over arbitrary base schemes.
\vspace{.05in}
\begin{proposition}
    $H(C)$ form an abelian group.
\end{proposition}
\begin{proof}
This follows directly from  Theorem \ref{Picard}.
\end{proof}

\subsection*{Comparison with Our Notion of Similarity}
In contrast to Kneser's approach, our work considers \emph{similarity classes} of quadratic forms, where two binary quadratic forms  $ (M_1, q_1, N_1) $ and $ (M_2, q_2, N_2) $ over $ R $   are said to be \emph{similar} if there exists a pair $ (\phi, \mu_\phi) $  with $ R $-module isomorphisms 
$$
\phi : M_1 \to M_2 \quad \text{and} \quad \mu_\phi : N_1 \to N_2
$$  
such that the following diagram of $ R $-modules commutes:
$$
\begin{tikzcd}
M_1 \arrow[r, "q_1"] \arrow[d, "\phi"'] & N_1 \arrow[d, "\mu_\phi"] \\
M_2 \arrow[r, "q_2"] & N_2
\end{tikzcd}
$$
That is, the equality  
$$
\mu_\phi \circ q_1 = q_2 \circ \phi
$$  
holds.

Furthermore, a key proposition from Max-Albert Knus’s book \textit{Quadratic and Hermitian Forms over Rings} (\cite[Proposition 7.1.1, Chapter IV, §7]{knus}) asserts that if two binary quadratic forms $ q_1 $ and $ q_2 $ are similar, then the choice of a similarity induces a unique isomorphism between $ C_0(q_1) $ and $ C_0(q_2) $ as algebras. Moreover, there is a unique $R$-module isomorphism between $ C_1(q_1) $ and $ C_1(q_2) $ that respects the respective $ C_0(q_1) $ and $ C_0(q_2) $-module structures. For the corresponding arguments over a base scheme $ S $, we refer to Propositions~\ref{induceevenclifford} and  Proposition~\ref{semilinear}.

In the case of binary forms, we have $ C_1(q_1) = M_1 $ and $ C_1(q_2) = M_2 $, and the aforementioned $ R $-module isomorphism between $ C_1(q_1) $ and $ C_1(q_2) $ coincides precisely with $\phi $. The same reasoning applies in the setting of an arbitrary base scheme $S$.

Thus, in our setting, the isomorphism $ \phi $ between modules is  $ R $-linear and semilinear, not necessarily $ C $-linear, which is required in the definition of $H(C)$. The same remark holds in the relative case over a base scheme $S$ .

\subsection*{Bridging the Gap: Compatibility Conditions and Automorphisms}
To align our approach more closely with Kneser's classification in $ H(C) $, we introduce a \emph{rigidification} of our quadratic algebra structure. Specifically, we impose a condition that restricts the automorphism group of the even Clifford algebra to be trivial. 
Under this assumption, any $ R $-linear isomorphism $ \phi : M_1 \to M_2 $ that satisfies the semilinearity condition becomes automatically $ C $-linear. The same reasoning applies in the setting of an arbitrary base scheme $S$.

This rigidity effectively forces the isomorphism $ \phi $ to intertwine the actions of the even Clifford algebras $ C_0(M_1) $ and $ C_0(M_2) $, and hence reduces similarity to isomorphism in the sense of Kneser.

\subsection*{Conclusion}
Kneser's definition of isomorphism in $ H(C) $ inherently assumes \textit{$ C $-linearity}, which aligns with treating quadratic forms as modules over their even Clifford algebras. In our work, we begin with a broader notion - \textit{ similarity}, defined via \textit{$ R $-linear maps} satisfying a semi-linearity condition. However, when we rigidify the structure of the underlying even Clifford algebra to eliminate nontrivial automorphisms, this similarity condition becomes equivalent to Kneser’s isomorphism of triples in $ H(C) $.

Hence, under appropriate conditions — specifically, when the automorphism group of the even Clifford algebra is trivial — our compatible isomorphisms reduce precisely to the \textit{$ C $-module isomorphisms} used by Kneser. This shows that Theorem~\ref{Picard} can be seen as a generalization of his construction, with Kneser’s framework emerging as a special case when the base algebra $ C $ is sufficiently rigid.

\subsection{Gauss Composition over arbitrary base}
The composition of binary quadratic forms is a long-established topic in number theory. Since Gauss’s seminal work in his \emph{Disquisitiones Arithmeticae}~\cite{Gauss}, there have been numerous efforts to simplify and generalize this concept. Martin Kneser, in his influential work~\cite{KNESER}, introduced the notion of quadratic modules \(E\) as modules over their even Clifford algebra \(C_0(E)\), defining the composition law as a tensor product over \(C_0(E)\). Notably, Kneser succeeded in establishing this composition law for affine schemes without imposing any conditions on the characteristic of the base ring. In the spirit of \cite{KNESER}, the composition theory of quaternary quadratic forms is systematically investigated in \cite{quatcomp}. In this paper, we extend the work of Kneser~\cite{KNESER} in two significant directions:

\begin{itemize}
    \item We consider \emph{primitive binary quadratic forms} over an arbitrary base scheme \( S \).
    
    \item We allow these forms to take values in an \emph{arbitrary line bundle} over \( S \), rather than restricting to the trivial bundle.
\end{itemize}
Let $C$ be a given quadratic algebra over an arbitrary base scheme $S$. We denote by \( \widetilde{H}(C) \) the set of all similarity classes of primitive binary quadratic forms having generalized even Clifford algebras isomorphic to $C$.
\vspace{.05in}
\begin{proposition}\label{pregauss}
  The set $\widetilde{H}(C)$ admits the structure of an abelian group.
\end{proposition}
\begin{proof}

We define the group operation as follows. Let $(E_1, q_1, L_1)$ and $(E_2, q_2, L_2)$ be two primitive binary quadratic forms over the base scheme $S$. By Lemma~\ref{formsimuninorm}, the universal norm forms associated with the  pairs $(C_0(E_1, q_1, L_1), C_1(E_1, q_1, L_1) = E_1)$ and $(C_0(E_2, q_2, L_2), E_2)$ are similar to the forms $(E_1, q_1, L_1)$ and $(E_2, q_2, L_2)$, respectively. The group operation is then defined by:
\[
(C_0(E_1, q_1, L_1), E_1) \cdot (C_0(E_2, q_2, L_2), E_2) = (C, E_1 \otimes_C E_2),
\]
where $C$ is the given quadratic $\OS$-algebra isomorphic to both $C_0(E_1, q_1, L_1)$ and $C_0(E_2, q_2, L_2)$, and $E_1 \otimes_C E_2$ is the tensor product over $C$.

Associativity and commutativity are trivial consequences of the rules for tensor products. The identity element is $E=C$ with $q=n_C$, the norm. The inverse of a given class $(E, q, L)$ is given by the universal norm form associated with the pair $(C_0(E, q, L), E)$, where the natural $C_0(E, q, L)$-action on $E$ replaced by $(c, x)\mapsto \bar{c}x$, where $\bar{c}$ denotes the image of $c$ under the canonical involution of the generalized even  Clifford algebra $C_0(E, q, L)$. For a comprehensive treatment and detailed justification of these claims, see \cite[Proposition 4.25]{sohamthesis}. Thus, $\widetilde{H}(C)$ acquires the structure of an abelian group.
\end{proof}

\vspace{.05in}
\begin{theorem}\label{thegausscomp}
The set of all primitive binary quadratic forms over an arbitrary base scheme $S$ that have isomorphic generalized even Clifford algebras admits the structure of an abelian group.
\end{theorem}
\begin{proof}
    This result follows directly from Proposition \ref{pregauss}.
\end{proof}

These generalizations broaden the classical theory to a more flexible and geometrically robust setting. 
% Wood in her work \cite{Wood}, gave a  set-theoretical bijection with a disjoint union of quotient sets of Picard groups (in
% the primitive case). But she gave this bijection for \textit{linear} binary quadratic forms also her motivations being related to
% moduli problems, she gave neither a group law nor an isomorphism with some Picard group. Dallaporta, in his work \cite{dallaporta} define the group law but for \textit{linear} binary quadratic forms. Wood's \cite{Wood} and Dallaporta's \cite{dallaporta} approach was for \textit{linear} binary quadratic forms and they have not used Clifford point of view and in contrast, our main innovation lies precisely in adopting  Clifford-theoretic viewpoint. 

In her work \cite{Wood}, Wood established a set-theoretic bijection—restricted to the primitive case—between linear binary quadratic forms and a disjoint union of quotient sets of Picard groups. Motivated primarily by moduli-theoretic considerations, she did not formulate a group law nor did she identify an explicit isomorphism with any particular Picard group. Subsequently, Dallaporta \cite{dallaporta} introduced a group structure on certain classes of primitive binary quadratic forms; however, his construction is again restricted to the setting of \textit{linear} binary quadratic forms. Moreover, it requires the additional assumption that $2$ is not a zero divisor in the global sections $\Gamma(S, \OS)$, see \cite[Theorem 3.24]{dallaporta}. Both Wood and Dallaporta consider \textit{linear} binary quadratic forms, and  their frameworks operate entirely outside the Clifford algebra perspective whereas we are considering binary quadratic forms in the classical sense over an arbitrary base scheme $S$, incorporating the Clifford-theoretic viewpoint, which also allows for a more intrinsic and algebraically robust formulation of the theory. 

% and also their frameworks operate entirely outside the Clifford algebra perspective. In contrast, the principal innovation of our approach lies in systematically incorporating the Clifford-theoretic viewpoint, which allows for a more intrinsic and algebraically robust formulation of the theory.

Theorem \ref{thegausscomp} gives the composition law on similarity classes of primitive binary quadratic forms taking values in line bundles over the base scheme $S$.
\vspace{.05in}
\begin{rmk}
    The extension of Gauss composition to the setting of primitive binary quadratic forms over an arbitrary base scheme $S$ can be carried out without imposing any additional conditions on $S$. However, when relating Gauss composition to $\mathrm{Pic}(C)$, the Picard group of the quadratic algebra $C$, it is necessary to impose certain technical assumptions on the base scheme $S$.  Specifically, as detailed in Theorem \ref{Picard}, we required $2$ to be a non-zero divisor in the global sections $\Gamma(S, \OS)$, to get a well-defined set-theoretic map from $H(C)$ to $\mathrm{Pic}(C)$. Thus, while the composition law itself is unconditional, the existence of natural comparison with the Picard group depends crucially on the geometry of the base.
\end{rmk}

\bibliographystyle{plain}
\bibliography{mybib}

@incollection{highcomp,
 author = {Bhargava, Manjul},
 title = {Higher composition laws and applications},
 booktitle = {Proceedings of the international congress of mathematicians (ICM), Madrid, Spain, August 22--30, 2006. Volume II: Invited lectures},
 isbn = {978-3-03719-022-7},
 pages = {271--294},
 year = {2006},
 publisher = {Z{\"u}rich: European Mathematical Society (EMS)},
 language = {English},
 zbMATH = {5057398},
 Zbl = {1129.11049}
}

@article{quatcomp,
 author = {Kneser, M. and Knus, M.-A. and Ojanguren, M. and Parimala, R. and Sridharan, R.},
 title = {Composition of quaternary quadratic forms},
 fjournal = {Compositio Mathematica},
 journal = {Compos. Math.},
 issn = {0010-437X},
 volume = {60},
 pages = {133--150},
 year = {1986},
 language = {English},
 url = {https://eudml.org/doc/89801},
 zbMATH = {3989443},
 Zbl = {0612.10015}
}

@book {dual,
    AUTHOR = {Eisenbud, David and Harris, Joe},
     TITLE = {3264 and all that---a second course in algebraic geometry},
 PUBLISHER = {Cambridge University Press, Cambridge},
      YEAR = {2016},
     PAGES = {xiv+616},
      ISBN = {978-1-107-60272-4; 978-1-107-01708-5},
   MRCLASS = {14-01 (14C15 14M15 14N10)},
  MRNUMBER = {3617981},
MRREVIEWER = {Arnaud\ Beauville},
       DOI = {10.1017/CBO9781139062046},
       URL = {https://doi.org/10.1017/CBO9781139062046},
}

@article{traceablewood,
 author = {Erman, Daniel and Wood, Melanie Matchett},
 title = {Gauss composition for {{\(\mathbb{P}^1\)}}, and the universal {Jacobian} of the {Hurwitz} space of double covers},
 fjournal = {Journal of Algebra},
 journal = {J. Algebra},
 issn = {0021-8693},
 volume = {470},
 pages = {320--352},
 year = {2017},
 language = {English},
 doi = {10.1016/j.jalgebra.2016.08.036},
 zbMATH = {6647041},
 Zbl = {1352.14001}
}

@article{sohamthesis,
  title={Similarity Classification of Binary Quadratic Forms using Clifford Pairs with Applications to {Gauss} Composition on General Schemes},
  author={Mondal, Soham},
  journal={IIT Madras: dissertation},
  year={2025}
}

@article{chanconic,
  title={Conic bundles and Clifford algebras},
  author={Chan, Daniel and Ingalls, Colin},
  journal={Contemporary Math},
  volume={562},
  pages={53--76},
  year={2012}
}

@article{voightcharacterizing,
 author = {Voight, John},
 title = {Characterizing quaternion rings over an arbitrary base},
 fjournal = {Journal f{\"u}r die Reine und Angewandte Mathematik},
 journal = {J. Reine Angew. Math.},
 issn = {0075-4102},
 volume = {657},
 pages = {113--134},
 year = {2011},
 language = {English},
 doi = {10.1515/CRELLE.2011.054},
 zbMATH = {5947067},
 Zbl = {1229.11146}
}

@article{tevbal,
 author = {Venkata Balaji, Thiruvalloor Eesanaipaadi},
 title = {Line-bundle-valued ternary quadratic forms over schemes},
 fjournal = {Journal of Pure and Applied Algebra},
 journal = {J. Pure Appl. Algebra},
 issn = {0022-4049},
 volume = {208},
 number = {1},
 pages = {237--259},
 year = {2007},
 language = {English},
 doi = {10.1016/j.jpaa.2005.12.002},
 zbMATH = {5078568},
 Zbl = {1167.11014}
}

@book{parimalareduced,
 editor = {Jacob, William B. and Lam, Tsit-Yuen and Robson, Robert O.},
 title = {Recent advances in real algebraic geometry and quadratic forms. {Proceedings} of the {RAGSQUAD} year, {Berkeley}, {CA}, {USA}, 1990-1991},
 fseries = {Contemporary Mathematics},
 series = {Contemp. Math.},
 issn = {0271-4132},
 volume = {155},
 isbn = {0-8218-5154-3},
 year = {1994},
 publisher = {Providence, RI: American Mathematical Society},
 language = {English},
 doi = {10.1090/conm/155},
 zbMATH = {568744},
 Zbl = {0788.00051}
}

@article{caenepeel,
 author = {Caenepeel, Stefaan and Van Oystaeyen, Freddy},
 title = {Quadratic forms with values in invertible modules},
 fjournal = {\(K\)-Theory},
 journal = {\(K\)-Theory},
 issn = {0920-3036},
 volume = {7},
 number = {1},
 pages = {23--40},
 year = {1993},
 language = {English},
 doi = {10.1007/BF00962792},
 zbMATH = {227621},
 Zbl = {0787.13004}
}

@incollection{Gauss,
  title={Carl {Friedrich Gauss}, disquisitiones arithmeticae (1801)},
  author={Neumann, Olaf},
  booktitle={Landmark Writings in Western Mathematics 1640-1940},
  pages={303--315},
  year={2005},
  publisher={Elsevier}
}

@incollection{bichselknus,
 author = {Bichsel, W. and Knus, M.-A.},
 title = {Quadratic forms with values in line bundles},
 booktitle = {Recent advances in real algebraic geometry and quadratic forms. Proceedings of the RAGSQUAD year, Berkeley, CA, USA, 1990-1991},
 isbn = {0-8218-5154-3},
 pages = {293--306},
 year = {1994},
 publisher = {Providence, RI: American Mathematical Society},
 language = {English},
 zbMATH = {589363},
 Zbl = {0810.11023}
}

@article{Bhatt_2012, title={Derived splinters in positive characteristic}, volume={148}, DOI={10.1112/S0010437X12000309}, number={6}, journal={Compositio Mathematica}, author={Bhatt, Bhargav}, year={2012}, pages={1757–1786}}

@article{Auel_2014,
   title={Fibrations in complete intersections of quadrics, Clifford algebras, derived categories, and rationality problems},
   volume={102},
   ISSN={0021-7824},
   url={http://dx.doi.org/10.1016/j.matpur.2013.11.009},
   DOI={10.1016/j.matpur.2013.11.009},
   number={1},
   journal={Journal de Mathématiques Pures et Appliquées},
   publisher={Elsevier BV},
   author={Auel, Asher and Bernardara, Marcello and Bolognesi, Michele},
   year={2014},
   month=jul, pages={249–291} }

@article{Wood,
 author = {Wood, Melanie Matchett},
 title = {Gauss composition over an arbitrary base},
 fjournal = {Advances in Mathematics},
 journal = {Adv. Math.},
 issn = {0001-8708},
 volume = {226},
 number = {2},
 pages = {1756--1771},
 year = {2011},
 language = {English},
 doi = {10.1016/j.aim.2010.08.018},
 zbMATH = {5835547},
 Zbl = {1262.11049}
}

@article{dallaporta,
 author = {Dallaporta, William},
 title = {Recovering the {Picard} group of quadratic algebras from {Wood}'s binary quadratic forms},
 fjournal = {International Journal of Number Theory},
 journal = {Int. J. Number Theory},
 issn = {1793-0421},
 volume = {21},
 number = {4},
 pages = {739--767},
 year = {2025},
 language = {English},
 doi = {10.1142/S179304212550037X},
 zbMATH = {8014145}
}

@article {Lowrank,
    AUTHOR = {Voight, John},
     TITLE = {Rings of low rank with a standard involution},
   JOURNAL = {Illinois J. Math.},
  FJOURNAL = {Illinois Journal of Mathematics},
    VOLUME = {55},
      YEAR = {2011},
    NUMBER = {3},
     PAGES = {1135--1154},
      ISSN = {0019-2082,1945-6581},
   MRCLASS = {16W10 (11E81 16G30)},
  MRNUMBER = {3069299},
MRREVIEWER = {Rosali\ Brusamarello},
       URL = {http://projecteuclid.org/euclid.ijm/1369841800},
}

@book{knus,
 author = {Knus, Max-Albert},
 title = {Quadratic and {Hermitian} forms over rings},
 fseries = {Grundlehren der Mathematischen Wissenschaften},
 series = {Grundlehren Math. Wiss.},
 issn = {0072-7830},
 volume = {294},
 isbn = {3-540-52117-8},
 year = {1991},
 publisher = {Berlin etc.: Springer-Verlag},
 language = {English},
 zbMATH = {48964},
 Zbl = {0756.11008}
}

@article{auel2015surjectivity,
 author = {Auel, Asher},
 title = {Surjectivity of the total {Clifford} invariant and {Brauer} dimension},
 fjournal = {Journal of Algebra},
 journal = {J. Algebra},
 issn = {0021-8693},
 volume = {443},
 pages = {395--421},
 year = {2015},
 language = {English},
 doi = {10.1016/j.jalgebra.2015.06.043},
 zbMATH = {6485321},
 Zbl = {1391.11071}
}

@phdthesis{bichsel1985quadratische,
  title={Quadratische raeume mit werten in invertierbaren moduln},
  author={Bichsel, Walter},
  year={1985},
  school={ETH Zurich}
}

@article{KNESER,
title = {Composition of binary quadratic forms},
journal = {Journal of Number Theory},
volume = {15},
number = {3},
pages = {406-413},
year = {1982},
issn = {0022-314X},
doi = {https://doi.org/10.1016/0022-314X(82)90041-5},
url = {https://www.sciencedirect.com/science/article/pii/0022314X82900415},
author = {Martin Kneser}
}

\end{document}